\documentclass[12pt]{article}
\usepackage{amsmath,rotating}
\usepackage{amsfonts}
\usepackage{amsthm}
\usepackage{vmargin}
\setmarginsrb{12mm}{10mm}{20mm}{22mm}{0pt}{0mm}{0pt}{10mm}
\begin{document}
\setlength{\baselineskip}{18pt}
\newtheorem{thm}{Theorem}[section]
\newtheorem{ass}{Assumption}[section]
\newtheorem{example}[thm]{Example}
\newtheorem{cor}[thm]{Corollary}
\newtheorem{notation}[thm]{Notation}
\newtheorem{notations}[thm]{Notations}
\newtheorem{definition}[thm]{Definition}
\newtheorem{lemma}[thm]{Lemma}
\newtheorem{claim}[thm]{Claim}
\newtheorem{mthm}[thm]{Meta-Theorem}
\newtheorem{prop}[thm]{Proposition}
\newtheorem{rem}[thm]{Remark}
\newtheorem{conj}[thm]{Conjecture}
\newtheorem{rems}[thm]{Remarks}
\newtheorem{theorem}{Theorem}[section]
\newtheorem{lem}{Lemma}[section]
\newtheorem{prp}[theorem]{Proposition}

\newtheorem{dfn}[theorem]{Definition}
\newtheorem{exmp}[theorem]{Example}
\newtheorem{remark}{Remark}
\newtheorem{summary}[theorem]{Summary}
\newtheorem{Hypothesis}[theorem]{Hypothesis}
\newtheorem{con}{Condition}[section]

\renewcommand {\theequation}{\thesection.\arabic{equation}}
\newcommand{\x}{{\bf x}}

\def\al{{\alpha}}\def\be{{\beta}}\def\de{{\delta}}
\def\ep{{\epsilon}}\def\ga{{\gamma}}
\def\la{{\lambda}}\def\om{{\omega}}\def\si{{\sigma}}
\def\th{{\theta}}\def\ze{{\zeta}}

\def\De{{\Delta}}\def\Ga{{\Gamma}}
\def\La{{\Lambda}}\def\Om{{\Omega}}\def\Si{{\Sigma}}

\def\tY{{\tilde{Y}}}\def\txi{{\tilde{\xi}}}\def\hxi{{\hat{\xi}}}

\newfam\msbmfam\font\tenmsbm=msbm10\textfont
\msbmfam=\tenmsbm\font\sevenmsbm=msbm7
\scriptfont\msbmfam=\sevenmsbm\def\bb#1{{\fam\msbmfam #1}}
\def\bA{\mathbf A}\def\BB{\mathbb B}\def\CC{\mathbb C}
\def\EE{\mathbb E}\def\HH{\mathbb H}\def\KK{\mathbb K}\def\PP{\mathbb P}\def\QQ{\mathbb Q}
\def\RR{\mathbb R}\def\ZZ{\mathbb Z}

\def\cA{{\cal A}}\def\cB{{\cal B}}\def\cC{{\cal C}}
\def\cD{{\cal D}}\def\cF{{\cal F}}\def\cG{{\cal G}}\def\cH{{\cal H}}
\def\cI{{\cal I}}\def\cL{{\cal L}}\def\cM{{\cal M}}
\def\cP{{\cal P}}\def\cS{{\cal S}}\def\cT{{\cal T}}\def\cV{{\cal V}}
\def\cX{{\cal X}}\def\cY{{\cal Y}}\def\cZ{{\cal Z}}

\def\vX{\vec{X}}\def\bY{\mathbf Y}\def\tY{{\tilde{Y}}}\def\hY{{\hat{Y}}}\def\chY{{\breve{Y}}}
\def\hla{{\hat{\lambda}}}\def\tR{{\tilde{R}}}\def\tP{{\tilde{\mathbb{P}}}}\def\chR{{\breve{R}}}
\def\hmu{{\hat{\mu}}} \def\hsi{{\hat{\sigma}}}\def\tze{{\tilde{\zeta}}}
\def\hze{{\hat{\zeta}}}

\def\<{\left<}\def\>{\right>}\def\ka{\kappa}

\def\({\left(}\def\){\right)}

\title{Large deviations for locally monotone stochastic partial differential equations driven by L\'evy noise \thanks{This work was started
when the second author visited the first author at the Department of
Mathematics, University of Macau.
 JX's research is supported by Macao Science and Technology Fund FDCT
076/2012/A3 and Multi-Year Research Grants of the University of Macau Nos.
MYRG2014-00015-FST and MYRG2014-00034-FST. JZ's research is supported by National Natural Science Foundation of China (NSFC) (No. 11431014, No. 11401557), and the Fundamental Research Funds for the Central Universities (No. WK 0010000048).}  }
\medskip

\author{ Jie Xiong\thanks{Department of Mathematics, Faculty of Sci. \& Tech., University of Macau, Macau.
  Email: {\tt jiexiong@umac.mo }} and Jianliang Zhai\thanks{School of Mathematical Sciences, University of Science and Technology of China, Hefei, 230026, China Email: {\tt zhaijl@ustc.edu.cn} }
}
%
\maketitle

\begin{abstract}
In this paper, we establish a large deviation principle for a type of stochastic partial differential equations (SPDEs) with locally monotone coefficients driven by L\'evy noise. The weak convergence method plays an important role.


2010 AMS Classification:  Primary 60H15.  Secondary 35R60, 37L55.
\end{abstract}

\section{Introduction}
\setcounter{equation}{0}
\renewcommand{\theequation}{\thesection.\arabic{equation}}

We shall prove via the weak convergence approach \cite{Budhiraja-Chen-Dupuis, Budhiraja-Dupuis-Maroulas., Dupuis-Ellis} the Freidlin-Wentzell type large deviation principle (LDP) for a family of locally monotone stochastic partial differentia equations (SPDEs) driven by L\'evy processes, these SPDEs include stochastic reaction-diffusion equations, stochastic Burgers type equations, stochastic 2D Navier-Stokes equations
and stochastic equations of non-Newtonian fluids. \vskip 0.2cm

Let $V$ be a reflexive and separable Banach space, which is densely and
continuously injected in a separable Hilbert space $(H,\ \langle\cdot,\cdot\rangle_H)$. Identifying $H$ with its dual we get
$$
V\subset H\cong H^*\subset V^*,
$$
where the star `*' denotes the dual spaces. Denote $\langle\cdot,\cdot\rangle_{V^*,V}$ the duality between $V^*$ and $V$, then we have
$$
\langle u,v\rangle_{V^*,V}=\langle u,v\rangle_{H},\ \ \ \forall\ u\in H,\ v\in V.
$$

Fix $T>0$ and let $(\Omega,\mathcal{F},(\mathcal{F}_t)_{t\in[0,T]},\mathbb{P})$ be a complete separable filtration probability space. Let $\mathcal{P}$ be the predictable $\sigma$-field, that is the $\sigma$-field on $[0,T]\times \Omega$ generated by all left continuous and $\mathcal{F}_t$-adapted
real-valued processes. Further denote by $\mathcal{BF}$ the $\sigma$-field of the progressively measurable sets on $[0,T]\times\Omega$, i.e.
$$
\mathcal{BF}=\{
              O\subset[0,T]\times\Omega:\ \forall t\in[0,T],\ O\cap([0,t]\times\Omega)\in\mathcal{B}([0,t])\otimes\mathcal{F}_t
            \},
$$
where $\mathcal{B}([0,t])$ denotes the Borel $\sigma$-field on $[0,t]$.

Now we consider the following type of SPDEs driven by L\'evy processes:
\begin{eqnarray}\label{eq SPDE 01}
&&dX^\epsilon_t=\mathcal{A}(t,X^\epsilon_t)dt+\epsilon\int_\mathbb{X}f(t,X^\epsilon_{t-},z)\widetilde{N}^{\epsilon^{-1}}(dt,dz),\\
&&X^\epsilon_0=x\in H,\nonumber
\end{eqnarray}
where $\mathcal{A}:[0,T]\times V\rightarrow V^*$ is a $\mathcal{B}([0,T])\otimes\mathcal{B}(V)$-measurable function.
$\mathbb{X}$ is a locally compact Polish space.
$N^{\epsilon^{-1}}$ is a Poisson random measure on $[0,T]\times\mathbb{X}$ with a
$\sigma$-finite
mean measure $\epsilon^{-1}\lambda_T\otimes \nu$, $\lambda_T$ is the Lebesgue measure on $[0,T]$ and $\nu$ is a $\sigma$-finite measure on $\mathbb{X}$.
\[\widetilde{N}^{\epsilon^{-1}}([0,t]\times B)=N^{\epsilon^{-1}}([0,t]\times B)-\epsilon^{-1}t\nu(B),\qquad
\forall B\in \mathcal{B}(\mathbb{X})\mbox{ with }\nu(B)<\infty,\] is the compensated Poisson random measure.
$f:[0,T]\times V\times \mathbb{X}\rightarrow H$ is a $\mathcal{B}([0,T])\otimes\mathcal{B}(V)\otimes\mathcal{B}(\mathbb{X})$-measurable function.

The following assumptions are from \cite{B-Liu-Zhu}, which guarantee that Eq. \eqref{eq SPDE 01} admits a unique solution.
Suppose that there exists constants $\alpha>1,\ \beta\geq0,\ \theta>0,\ C>0$, positive functions $K$ and $F$ and a function $\rho:V\rightarrow [0,+\infty)$ which is measurable and bounded on the balls,
 such that the following conditions hold for all $v,v_1,v_2\in V$ and $t\in [0,T]$:
\begin{itemize}
\item[{\bf (H1)}](Hemicontinuity) The map $s\mapsto \langle \mathcal{A}(t,v_1+s v_2),v\rangle_{V^*,V}$ is continuous on $\mathbb{R}$.

\item[{\bf (H2)}](Local monotonicity)
\begin{eqnarray*}
&&2\langle \mathcal{A}(t,v_1)-\mathcal{A}(t,v_2),v_1-v_2\rangle_{V^*,V}+\int_{\mathbb{X}}\|f(t,v_1,z)-f(t,v_2,z)\|^2_H\nu(dz)\\
&\leq&
(K_t+\rho(v_2))\|v_1-v_2\|^2_H,
\end{eqnarray*}

\item[{\bf (H3)}](Coercivity)
$$
2\langle \mathcal{A}(t,v),v\rangle_{V^*,V}+\theta\|v\|^\alpha_V
\leq
F_t(1 + \|v\|^2_H).
$$

\item[{\bf (H4)}](Growth)
$$
\|\mathcal{A}(t,v)\|^{\frac{\alpha}{\alpha-1}}_{V^*}
\leq
(F_t+C\|v\|^\alpha_V)(1+\|v\|_H^\beta).
$$

\end{itemize}

\begin{definition}\label{def 02}
 An $H$-valued c\'adl\'ag $\mathcal{F}_t$-adapted process $\{X^\epsilon_t\}_{t\in[0,T]}$ is called a solution of Eq. \eqref{eq SPDE 01}, if for its $dt\times \mathbb{P}$-equivalent class $\widehat{X}^\epsilon$ we have
\begin{itemize}
\item[(1)] $\widehat{X}^\epsilon\in L^\alpha([0,T];V)\cap L^2([0,T];H)$, $\mathbb{P}$-a.s.;

\item[(2)] the following equality holds $\mathbb{P}$-a.s.:
\begin{eqnarray*}
X^\epsilon_t=x+\int_0^t \mathcal{A}(s,\overline{X}^\epsilon_s)ds+\epsilon\int_0^t\int_{\mathbb{X}}f(s,\overline{X}^\epsilon_{s},z)\widetilde{N}^{\epsilon^{-1}}(ds,dz),\ \ t\in[0,T],
\end{eqnarray*}
\end{itemize}
where $\overline{X}^\epsilon$ is any $V$-valued progressively measurable $dt\times \mathbb{P}$ version of $\widehat{X}^\epsilon$.
\end{definition}

With a minor modification of
\cite[Theorem 1.2]{B-Liu-Zhu}, we have the following existence and uniqueness theorem for the solution of Eq. \eqref{eq SPDE 01}.

\begin{thm}\label{thm solution 01}
Suppose that conditions $(H1)$-$(H4)$ hold for $F, K\in L^1([0,T];\mathbb{R}^+)$, and there exists a constant $\gamma<\frac{\theta}{2\beta}$ and $G\in L^1([0,T];\mathbb{R}^+)$
such that for all $t\in[0,T]$ and $v\in V$ we have
\begin{eqnarray}\label{eq solution 01}
\int_{\mathbb{X}}\|f(t,v,z)\|^2_H\nu(dz)
\leq
F_t(1+\|v\|^2_H)+\gamma\|v\|^\alpha_V;
\end{eqnarray}

\begin{eqnarray}\label{eq solution 02}
\int_{\mathbb{X}}\|f(t,v,z)\|^{\beta+2}_H\nu(dz)
\leq
G_t(1+\|v\|^{\beta+2}_H);
\end{eqnarray}

\begin{eqnarray}\label{eq solution 03}
\rho(v)\leq C(1+\|v\|_V^\alpha)(1+\|v\|^\beta_H).
\end{eqnarray}
Then

\begin{itemize}
\item[(1)] For any $x\in L^{\beta+2}(\Omega,\mathcal{F}_0,\mathbb{P};H)$, (\ref{eq SPDE 01}) has a unique solution $\{X^\epsilon_t\}_{t\in[0,T]}$.

\item[(2)] If $\gamma$ is small enough, then
$$
\mathbb{E}\Big(\sup_{t\in[0,T]}\|X^\epsilon_t\|_H^{\beta+2}\Big)+\mathbb{E}\int_0^T\|X^\epsilon_t\|_H^\beta\|X^\epsilon_t\|^\alpha_Vdt
\leq
C_\epsilon\Big(\mathbb{E}\|x\|_H^{\beta+2}+\int_0^TG_tdt+\int_0^TF_tdt\Big).
$$
\end{itemize}
\end{thm}

Our aim in the present paper is to establish a LDP for the solution of (\ref{eq SPDE 01}) as $\epsilon\rightarrow 0$ on
$D([0,T],H)$, the space of $H$-valued $\rm c\grave{a}dl\grave{a}g$ functions on $[0,T]$.

In the past three decades, there are numerous literatures about the LDP for stochastic evolution equations (SEEs)
and SPDEs driven by Gaussian processes (cf. \cite{Bessaih-Millet, Budhiraja-Dupuis, B-D-M-1., CW, CR, CM2, C, Duan-Millet, Manna-Sritharan-Sundar, Liu, Ren-Zhang, Rockner-Zhang-Zhang, S, Z, Zhang}, etc.). Many of these results were obtained by using the weak convergence approach for the case of Gaussian noise, introduced by \cite{Budhiraja-Dupuis, B-D-M-1.}, see, for example, \cite{Bessaih-Millet, Budhiraja-Dupuis, B-D-M-1., Duan-Millet, Manna-Sritharan-Sundar, Liu, Ren-Zhang, Rockner-Zhang-Zhang, Zhang}. This approach has been proved to be very effective for various finite/infinite-dimensional stochastic dynamical systems. One of the main advantages of this approach is that one only needs to make
some necessary moment estimates.

The situations for SEEs and SPDEs driven by L\'evy noise are drastically different because of the appearance of the jumps. There are only
a few results on this topic so far. The first paper on LDP for SEEs of jump type is $\rm R\ddot{o}kner$ and Zhang \cite{Rockner-Zhang} where
the additive noise is considered. The study of LDP for multiplicative L\'evy noise has been carried out as well, e.g., \cite{Swiech-Zabczyk} and \cite{Budhiraja-Chen-Dupuis} for SEEs where the LDP was established on a larger space (hence, with a weaker topology) than the actual state space
of the solution, \cite{YZZ} for SEEs on the actual state space, \cite{ZZ} for the 2-D stochastic Navier-Stokes equations (SNSEs). Before \cite{ZZ},
Xu and Zhang\cite{Xu-Zhang} dealt with  the 2-D SNSEs driven by additive L\'evy noise. We also refer to \cite{A1, A2, Bao-Yuan, Feng-Kurtz} for related results.

To obtain our result, we will use the weak convergence approach introduced by \cite{Budhiraja-Chen-Dupuis, Budhiraja-Dupuis-Maroulas., Dupuis-Ellis}
for the case of Poisson random measures. This approach is a powerful tool to prove the LDP for SEEs and SPDEs driven by L\'evy noise, which has been
applied for several dynamical systems. The weak convergence method was first used in \cite{Budhiraja-Chen-Dupuis} to obtain LDP for SPDEs on co-nuclear spaces driven by $L\acute{e}vy$ noises and in \cite{YZZ} for SPDEs on Hilbert spaces with regular coefficients. Paper \cite{ZZ} deals with the 2-D SNSEs driven by multiplicative L\'evy noise.
Bao and Yuan \cite{Bao-Yuan} established a LDP for a class of stochastic functional differential equations of neutral type driven by a finite-dimensional Wiener process and a stationary Poisson random measure.

Monotone method is a main tool to prove the existence and uniqueness of SPDEs, and it can tackle a large class of SPDEs, for more details, see \cite{B-Liu-Zhu, Liu-Rockner} and references therein. Working in the framework of \cite{B-Liu-Zhu}, the purpose of this paper is to establish a LDP for a family of locally monotone SPDEs (\ref{eq SPDE 01}) driven by pure jumps. In addition to the difficulties caused by the jumps, much of our problem is to deal with the monotone operator $\mathcal{A}$. Using the weak convergence approach, the main point is to prove the tightness of some controlled SPDEs, see (\ref{eq SPDE 02}). This is highly nontrivial. We first divide the controlled SPDEs (\ref{eq SPDE 02}) into three parts, and establish the tightness of each part in suitable
larger space, respectively, see Proposition \ref{lem 4.6}. And then via the Skorohod representation theorem we are able to show the weak convergence
actually takes place in the spae $D([0,T],H)$. Finally, we mention that our framework can tackle the SPDEs with some polynomial growth, see Example 4.3 in \cite{B-Liu-Zhu}.

This paper is organized as follows. In Section 2, we will recall the abstract criteria for LDP obtained in \cite{Budhiraja-Chen-Dupuis, Budhiraja-Dupuis-Maroulas.}. In Section 3, we will show the main result of this paper. Section 4 and Section 5 is devoted to prove prior results on the controlled SPDEs (\ref{eq SPDE 02}), which play a key role in this paper. The entire Section 6 is to establish the LDP for (\ref{eq SPDE 01}).


\section{Preliminaries}
\setcounter{equation}{0}
\renewcommand{\theequation}{\thesection.\arabic{equation}}

\subsection{Poisson Random Measure}\label{Section Representation}

For  convenience of the reader, we shall adopt the notation in \cite{Budhiraja-Chen-Dupuis} and \cite{Budhiraja-Dupuis-Maroulas.}.
Recall that $\mathbb{X}$ is a locally compact Polish space. Denote by $\mathcal{M}_{FC}(\mathbb{X})$ the collection of all measures on $(\mathbb{X},\mathcal{B}(\mathbb{X}))$ such that $\nu(K)<\infty$ for any compact $K \in \mathcal{B}(\mathbb{X})\}$.
Denote by $C_c(\mathbb{X})$ the space of continuous functions with compact supports, endow $\mathcal{M}_{FC}(\mathbb{X})$ with the weakest topology such that for every $f\in
C_c(\mathbb{X})$, the function
\[\nu\rightarrow\langle f,\nu\rangle=\int_{\mathbb{X}}f(u)d\nu(u),\]
is continuous for $\nu\in\mathcal{M}_{FC}(\mathbb{X})$.
This topology can be metrized such that $\mathcal{M}_{FC}(\mathbb{X})$ is a Polish space (see e.g.
\cite{Budhiraja-Dupuis-Maroulas.}).

Fixing $T\in(0,\infty)$, we denote $\mathbb{X}_T=[0,T]\times\mathbb{X}$ and $\nu_T=\lambda_T\otimes\nu$ with $\lambda_T$ being Lebesgue measure on $[0,T]$ and
$\nu\in\mathcal{M}_{FC}(\mathbb{X})$.
Let $\textbf{n}$ be a Poisson random measure on $\mathbb{X}_T$ with intensity measure
$\nu_T$, it is well-known \cite{Ikeda-Watanabe} that $\textbf{n}$ is an $\mathcal{M}_{FC}(\mathbb{X}_T)$ valued random variable such that

(i) for each
$B\in\mathcal{B}(\mathbb{X}_T)$
with $\nu_T(B)<\infty$, $\textbf{n}(B)$ is Poisson distributed with mean $\nu_T(B)$;

(ii) for disjoint
$B_1,\cdots,B_k\in\mathcal{B}(\mathbb{X}_T)$, $\textbf{n}(B_1),\cdots,\textbf{n}(B_k)$ are mutually independent
random
variables.

For notational simplicity, we write from now on
\begin{eqnarray}\label{eq P4 star}
\mathbb{M}=\mathcal{M}_{FC}(\mathbb{X}_T),
\end{eqnarray}
and denote by $\mathbb{P}$ the probability measure induced by $\textbf{n}$ on $(\mathbb{M},\mathcal{B}(\mathbb{M}))$.
Under $\mathbb{P}$, the canonical map, $N:\mathbb{M}\rightarrow\mathbb{M},\ N(m)\doteq m$, is a Poisson random measure with
intensity measure $\nu_T$. With applications to large deviations in mind, we also consider, for $\theta>0$,
probability
measures $\mathbb{P}_\theta$ on $(\mathbb{M},\mathcal{B}(\mathbb{M}))$ under which $N$ is a Poisson random measure
with intensity $\theta\nu_T$. The corresponding expectation operators will be denoted by $\mathbb{E}$ and
$\mathbb{E}_\theta$,
respectively.

For further use, simply denote
\begin{equation}\label{eq0302a}
\mathbb{Y}=\mathbb{X}\times[0,\infty), \ \ \mathbb{Y}_T=[0,T]\times\mathbb{Y}, \ \ \bar{\mathbb{M}}=\mathcal{M}_{FC}(\mathbb{Y}_T).\end{equation}
Let $\bar{\mathbb{P}}$ be the unique probability measure on $(\bar{\mathbb{M}},\mathcal{B}(\bar{\mathbb{M}}))$
under which the canonical map, $\bar{N}:\bar{\mathbb{M}}\rightarrow\bar{\mathbb{M}},\bar{N}(\bar{m})\doteq \bar{m}$, is a Poisson
random
measure with intensity measure $\bar{\nu}_T=\lambda_T\otimes\nu\otimes \lambda_\infty$, with $\lambda_\infty$ being
Lebesgue measure on $[0,\infty)$.
The corresponding expectation operator will be denoted by $\bar{\mathbb{E}}$. Let
$\mathcal{F}_t\doteq\sigma\{\bar{N}((0,s]\times A):0\leq s\leq t,A\in\mathcal{B}(\mathbb{Y})\},$ and let
$\bar{\mathcal{F}}_t$
denote the completion under $\bar{\mathbb{P}}$. We denote by $\bar{\mathcal{P}}$ the predictable $\sigma$-field on
$[0,T]\times\bar{\mathbb{M}}$
with the filtration $\{\bar{\mathcal{F}}_t:0\leq t\leq T\}$ on $(\bar{\mathbb{M}},\mathcal{B}(\bar{\mathbb{M}}))$.
Let
$\bar{\mathbb{A}}$
be the class of all $(\bar{\mathcal{P}}\otimes\mathcal{B}(\mathbb{X}))/\mathcal{B}([0,\infty))$-measurable maps
$\varphi:\mathbb{X}_T\times\bar{\mathbb{M}}\rightarrow[0,\infty)$. For $\varphi\in\bar{\mathbb{A}}$, we shall suppress the argument $\bar m$ in
$\varphi(s,x,\bar m)$ and simply write $\varphi(s,x)=\varphi(s,x,\bar m)$.
  Define a
counting process $N^\varphi$ on
$\mathbb{X}_T$ by
   \begin{eqnarray}\label{Jump-representation}
      N^\varphi((0,t]\times U)=\int_{(0,t]\times U}\int_{(0,\infty)}1_{[0,\varphi(s,x)]}(r)\bar{N}(dsdxdr),\
      t\in[0,T],U\in\mathcal{B}(\mathbb{X}).
   \end{eqnarray}

\noindent The above $N^\varphi$ is called a controlled random measure, with $\varphi$ selecting the intensity for the points at location
$x$
and time $s$, in a possibly random but non-anticipating way. When $\varphi(s,x,\bar{m})\equiv\theta\in(0,\infty)$, we
write $N^\varphi=N^\theta$. Note that $N^\theta$ has the same distribution with respect to $\bar{\mathbb{P}}$ as $N$
has with respect to $\mathbb{P}_\theta$.
\vskip 3mm
Define $l:[0,\infty)\rightarrow[0,\infty)$ by
    $$
    l(r)=r\log r-r+1,\ \ r\in[0,\infty).
    $$
For any $\varphi\in\bar{\mathbb{A}}$ the quantity
    \begin{eqnarray}\label{L_T}
      L_T(\varphi)=\int_{\mathbb{X}_T}l(\varphi(t,x,\omega))\nu_T(dtdx)
    \end{eqnarray}
is well defined as a $[0,\infty]$-valued random variable.

%
%

\subsection{A general criterion for large deviation principle \cite[Theorem 4.2]{Budhiraja-Dupuis-Maroulas.}}

We first state the large deviation principle we are concerned with.
Let $\{X^\epsilon,\epsilon>0\}\equiv\{X^\epsilon\}$ be a family of random variables defined on a probability space
$(\Omega,\mathcal{F},\mathbb{P})$
and taking values in a Polish space $\mathcal{E}$. Denote the expectation with
respect to $\mathbb{P}$ by $\mathbb{E}$.
The theory of large deviations is concerned with events $A$ for which probability $\mathbb{P}(X^\epsilon\in A)$
converges to
zero exponentially fast as $\epsilon\rightarrow 0$. The exponential decay rate of such probabilities is typically
expressed
in terms of a 'rate function' $I$ defined as below.
    \begin{definition}\label{Dfn-Rate function}
       \emph{\textbf{(Rate function)}} A function $I: \mathcal{E}\rightarrow[0,\infty]$ is called a rate function on
       $\mathcal{E}$,
       if for each $M<\infty$ the level set $\{y\in\mathcal{E}:I(y)\leq M\}$ is a compact subset of $\mathcal{E}$.
       For
       $A\in \mathcal{B}(\mathcal{E})$,
       we define $I(A)\doteq\inf_{y\in A}I(y)$.
    \end{definition}
    \begin{definition}  \label{d:LDP}
       \emph{\textbf{(Large deviation principle)}} Let $I$ be a rate function on $\mathcal{E}$. The sequence
       $\{X^\epsilon\}$
       is said to satisfy a large deviation principle (LDP) on $\mathcal{E}$ with rate function $I$ if the following two
       conditions
       hold.

         a. LDP upper bound. For each closed subset $F$ of $\mathcal{E}$,
              $$
                \limsup_{\epsilon\rightarrow 0}\epsilon\log\mathbb{P}(X^\epsilon\in F)\leq-I(F).
              $$

         b. LDP lower bound. For each open subset $G$ of $\mathcal{E}$,
              $$
                \liminf_{\epsilon\rightarrow 0}\epsilon\log\mathbb{P}(X^\epsilon\in G)\geq-I(G).
              $$
    \end{definition}
\vskip 3mm

Next, we recall the general criterion for large deviation principles established in \cite{Budhiraja-Dupuis-Maroulas.}. Let
$\{\mathcal{G}^\epsilon\}_{\epsilon>0}$
be a family of measurable maps from $\mathbb{M}$ to $\mathbb{U}$, where $\mathbb{M}$ is introduced in (\ref{eq P4 star}) and $\mathbb{U}$ is a Polish space. We present
below a sufficient condition for LDP of the family
$Z^\epsilon=\mathcal{G}^\epsilon\Big(\epsilon N^{\epsilon^{-1}}\Big)$,
as $\epsilon\rightarrow 0$.

Define
   \begin{eqnarray}\label{S_N}
     S^N=\Big\{g:\mathbb{X}_T\rightarrow[0,\infty):\,L_T(g)\leq N\Big\},
   \end{eqnarray}
a function $g\in S^N$ can be identified with a measure $\nu_T^g\in\mathbb{M}$, defined by
   \begin{eqnarray*}
      \nu_T^g(A)=\int_A g(s,x)\nu_T(dsdx),\ \ A\in\mathcal{B}(\mathbb{X}_T).
   \end{eqnarray*}
This identification induces a topology on $S^N$ under which $S^N$ is a compact space, see the Appendix of
\cite{Budhiraja-Chen-Dupuis}.
Throughout this paper we use this topology on $S^N$. Denote $S=\cup_{N=1}^\infty S^N$ and $\bar{\mathbb{A}}^N:=\{\varphi\in\bar{\mathbb{A}}\ \text{and}\ \varphi(\omega)\in S^N,\ \bar{\mathbb{P}}\text{-}a.s.\}$.


\begin{con}\label{LDP}
 There exists a measurable map $\mathcal{G}^0:\mathbb{M}\rightarrow \mathbb{U}$ such that the following hold.

 a). For all $ N\in\mathbb{N}$, let $g_n,\ g\in S^N$ be such that $g_n\rightarrow g$ as
 $n\rightarrow\infty$. Then
      $$
         \mathcal{G}^0\Big(\nu_T^{g_n}\Big)\rightarrow \mathcal{G}^0\Big(\nu_T^{g}\Big)\quad\text{in}\quad \mathbb{U}.
      $$

 b). For all $N\in\mathbb{N}$, let $\varphi_\epsilon, \varphi\in\bar{\mathbb{A}}^N$ be
 such that $\varphi_\epsilon$
 converges in distribution to $\varphi$ as $\epsilon\rightarrow 0$. Then
      $$
         \mathcal{G}^\epsilon\Big(\epsilon
         N^{\epsilon^{-1}\varphi_\epsilon}\Big)\Rightarrow
         \mathcal{G}^0\Big(\nu_T^{\varphi}\Big).
      $$
\end{con}
In this paper, we use the symbol ``$\Rightarrow$" to denote convergence in distribution.

For $\phi\in\mathbb{U}$, define $\mathbb{S}_\phi=\Big\{g\in S:\phi=\mathcal{G}^{0}\Big(
\nu^g_T\Big)\Big\}$. Let $I:\mathbb{U}\rightarrow[0,\infty]$
be defined by
     \begin{eqnarray}\label{Rate function I}
        I(\phi)=\inf_{g\in\mathbb{S}_\phi}L_T(g),\ \qquad \phi\in\mathbb{U}.
     \end{eqnarray}
By convention, $I(\phi)=\infty$ if $\mathbb{S}_\phi=\emptyset$. The following criterion for LDP was established in Theorem 4.2 of \cite{Budhiraja-Dupuis-Maroulas.}.

\begin{thm}\label{LDP-main}
For $\epsilon>0$, let $Z^\epsilon$ be defined by $Z^\epsilon=\mathcal{G}^\epsilon\Big(\epsilon
N^{\epsilon^{-1}}\Big)$, and suppose
that Condition \ref{LDP} holds. Then
the family $\{Z^\epsilon\}_{\epsilon>0}$ satisfies a large deviation principle with the rate function $I$ defined by \eqref{Rate function I}.
\end{thm}

For applications, the following strengthened form of Theorem \ref{LDP-main} is more useful and was established in Theorem 2.4
of \cite{Budhiraja-Chen-Dupuis}. Let $\{K_n\subset \mathbb{X},\
n=1,2,\cdots\}$
be an increasing sequence of compact sets such that $\cup _{n=1}^\infty K_n=\mathbb{X}$. For each $n$, let
   \begin{eqnarray*}
     \bar{\mathbb{A}}_{b,n}
          &=&
              \Big\{\varphi\in\bar{\mathbb{A}}:
                                   {\rm \ for\ all\ }(t,\omega)\in[0,T]\times\bar{\mathbb{M}},\
                                   n\geq\varphi(t,x,\omega)\geq
                                   1/n\ {\rm if}\ x\in K_n\ \\
                                  &&\ \ \ \ \ \ \ \ \ \ \ \ \
                                       {\rm and}\ \varphi(t,x,\omega)=1\ if\ x\in K_n^c
              \Big\},
   \end{eqnarray*}
and let $\bar{\mathbb{A}}_{b}=\cup _{n=1}^\infty\bar{\mathbb{A}}_{b,n}$. Define
$\tilde{\mathbb{A}}^N=\bar{\mathbb{A}}^N\cap\Big\{\phi: \phi\in\bar{\mathbb{A}}_b\Big\}$.

\begin{thm}{\bf\rm\cite{Budhiraja-Chen-Dupuis}}\label{LDP-main-01}
  Suppose Condition \ref{LDP} holds with $\bar{\mathbb{A}}^N$ therein replaced by $\tilde{\mathbb{A}}^N$. Then the conclusions of
  Theorem \ref{LDP-main}
  continue to hold .
\end{thm}


 \section{LDP for Eq. \eqref{eq SPDE 01}}
 \setcounter{equation}{0}
\renewcommand{\theequation}{\thesection.\arabic{equation}}

 Assume that $X_{0}=x\in H$ is deterministic. Let $X^\epsilon$ be the $H$-valued solution to Eq. (\ref{eq SPDE 01}) with initial value $x$. In this
 section, we state the LDP on $D([0,T],H)$ for $\{X^\epsilon\}$ under suitable assumptions.

 Take $\mathbb{U}=D([0,T],H)$ in Condition \ref{LDP} with the Skorokhod topology $\mathbb{U}_S$. We know that $(\mathbb{U},\ \mathbb{U}_S)$ is a Polish space. For $p>0$, define
 \begin{eqnarray*}
\mathcal{H}_p&=&\Big\{h:[0,T]\times\mathbb{X}\to\mathbb{R}^+:\ \exists\delta>0, s.t.\ \forall\Gamma\in\mathcal{B}([0,T])\otimes\mathcal{B}(\mathbb{X})\ with\ \nu_T(\Gamma)<\infty,\\
&&\qquad\qquad \mbox{we have } \ \int_\Gamma\exp(\delta h^p(t,y))\nu(dy)dt<\infty\Big\}.\label{Fun h}
\end{eqnarray*}

\begin{remark}\label{Remark 2}
It is easy to check that $\mathcal{H}_p\subset \mathcal{H}_{p'}$ for any $p'\in(0,p)$ and
$$
\Big\{h:[0,T]\times \mathbb{X}\to\mathbb{R}^+,\ \sup_{(t,y)\in[0,T]\times\mathbb{X}}h(t,y)<\infty\Big\}\subset\mathcal{H}_p,\ \ \forall p>0.
$$

\end{remark}

To study LDP of Eq. \eqref{eq SPDE 01}, besides the assumptions (H1)-(H4), we further need
\begin{itemize}
  \item[{\bf (H5)}] There exist $\eta_0>0$, $p\geq \Upsilon$ with $\Upsilon:=\frac{2\beta(\alpha-1)(\alpha+\eta_0)}{\alpha}\vee\frac{4(\alpha-1)(\alpha+\eta_0)}{\alpha}\vee 4\vee (\beta+2)$, and  $L_f\in L_2(\nu_T)\cap L_4(\nu_T)\cap L_{\beta+2}(\nu_T)\cap L_\Upsilon(\nu_T)\cap L_\frac{\Upsilon}{2}(\nu_T)\cap \mathcal{H}_p$ such that
  $$
  \|f(t,v,z)\|_H\leq L_f(t,z)(1+\|v\|_H),\ \ \forall (t,v,z)\in[0,T]\times V\times\mathbb{X}.
  $$

  \item[{\bf (H6)}]There exists $G_f\in L_2(\nu_T)\cap \mathcal{H}_2$ such that
  $$
  \|f(t,v_1,z)-f(t,v_2,z)\|_H\leq G_f(t,z)\|v_1-v_2\|_H,\ \ \forall (t,z)\in[0,T]\times\mathbb{X},\ \ v_1,v_2\in V.
  $$
\end{itemize}

\begin{remark} It is easy to check that
$$
L_{2}(\nu_T)\cap \Big\{h:[0,T]\times \mathbb{X}\to\mathbb{R}^+,\ \|h\|_\infty<\infty\Big\}
\subset
L_2(\nu_T)\cap L_4(\nu_T)\cap L_{\beta+2}(\nu_T)\cap L_\Upsilon(\nu_T)\cap L_\frac{\Upsilon}{2}(\nu_T)\cap \mathcal{H}_p,
$$
where $\|h\|_\infty=\sup_{(t,y)\in[0,T]\times\mathbb{X}}h(t,y)$.

\end{remark}
\vskip 0.3cm

It follows from Theorem \ref{thm solution 01} that, for every $\epsilon>0$, there exists a measurable map $\mathcal{G}^{\epsilon}$:
$\bar{\mathbb{M}}\rightarrow D([0,T]; H)$ such that, for any Poisson random measure ${\bf n}^{\epsilon^{-1}}$ on $[0,T]\times
\mathbb{X}$ with mean measure $\epsilon^{-1} \lambda_{T}\otimes \nu$ given on some probability space,
$\mathcal{G}^{\epsilon}(\epsilon {\bf n}^{\epsilon^{-1}})$ is the unique solution $X^\epsilon$ of
$(\ref{eq SPDE 01})$ with
$\widetilde{N}^{\epsilon^{-1}}$ replaced by ${\bf \widetilde{n}}^{\epsilon^{-1}}$, here ${\bf \widetilde{n}}^{\epsilon^{-1}}$ is the compensated
Poisson random measure of ${\bf n}^{\epsilon^{-1}}$.
\vskip 0.3cm

To state our main result, we need to introduce the map $\mathcal{G}^0$. Recall $S$ given in Section 2.2. For $g\in S$, consider the following deterministic PDE (the skeleton equation):
\begin{eqnarray*}
  X^{0,g}_t=x+\int_0^t\mathcal{A}(s,X^{0,g}_s)ds+\int_0^tf(s,X^{0,g}_s,z)(g(s,z)-1)\nu(dz)ds,\ \ in\ V^*.
\end{eqnarray*}
By Proposition \ref{lem 4.4} below, this equation has a unique solution $X^{0,g}\in C([0,T],H)\cap L^\alpha([0,T],V)$.
Define
\begin{eqnarray}\label{eq G0}
\mathcal{G}^0(\nu_T^g):=X^{0,g},\ \ \ \forall g\in S.
\end{eqnarray}
Let $I:\mathbb{U}=D([0,T],H)\rightarrow[0,\infty]$ be defined as in (\ref{Rate function I}). The following is the main result of this paper.
\begin{thm}\label{th main LDP}
 Assume that {\bf (H1)}-{\bf (H6)} and (\ref{eq solution 03}) hold. Then the family $\{X^\epsilon\}_{\epsilon>0}$ satisfies
an LDP on $D([0,T],H)$ with the rate function $I$ under the topology of uniform convergence.
\end{thm}
\begin{proof}
According to Theorem \ref{LDP-main-01}, we only need to verify Condition \ref{LDP}, which will be done in the last section.
\end{proof}

\section{Tightness of $\cG^\ep(\ep N^{\ep^{-1}\varphi_\ep})$}
\setcounter{equation}{0}
\renewcommand{\theequation}{\thesection.\arabic{equation}}

In this section, we first state three lemmas whose proofs can be adopted from those in \cite{Budhiraja-Chen-Dupuis}, \cite{YZZ} and \cite{Budhiraja-Dupuis-Maroulas.}. Then, we establish two key estimates for the stochastic processes studied in this paper. Finally, we prove the tightness of this family of these stochastic processes.

Using similar arguments as those in proving \cite[Lemma 3.4]{Budhiraja-Chen-Dupuis}, we can establish the following lemma.

\begin{lem}\label{Lemma-Condition-0,H-1,H}
For any $h\in\mathcal{H}_p\cap L_{p'}(\nu_T)$, $p'\in(0,p]$, there exists a constant $C_{h,p,p',N}$ such that
     \begin{eqnarray}\label{Inq-G-0-1}
        C_{h,p,p',N}:=\sup_{g\in S^N}\int_{\mathbb{X}_T}h^{p'}(s,v)(g(s,v)+1)\nu(dv)ds<\infty.
     \end{eqnarray}
For any $h\in\mathcal{H}_2\cap L_{2}(\nu_T)$, there exists a constant $C_{h,N}$ such that
     \begin{eqnarray}\label{Inq-G-0-2}
        C_{h,N}:=\sup_{g\in S^N}\int_{\mathbb{X}_T}h(s,v)|g(s,v)-1|\nu(dv)ds<\infty.
     \end{eqnarray}
\end{lem}

Using the argument used for proving \cite[Lemmas 3.4 and 3.11]{Budhiraja-Chen-Dupuis} and \cite[(3.19)]{YZZ}, we further get
\begin{lem}\label{lem-thm2-02}
Let
$h:\ \mathbb{X}_T\rightarrow\mathbb{R}$ be a measurable function such that
   \begin{eqnarray*}
     \int_{\mathbb{X}_T}|h(s,v)|^2\nu(dv)ds<\infty,
   \end{eqnarray*}
and for all $\delta\in(0,\infty)$
   \begin{eqnarray*}
     \int_{E}\exp(\delta|h(s,v)|)\nu(dv)ds<\infty,
   \end{eqnarray*}
for all $E\in\mathcal{B}(\mathbb{X}_T)$ satisfying $\nu_T(E)<\infty$.

a). Fix $N\in\mathbb{N}$, and let $g_n,g\in S^N$ be such that $g_n\rightarrow g$ as $n\rightarrow\infty$. Then
   \begin{eqnarray*}
     \lim_{n\rightarrow\infty}\int_{\mathbb{X}_T}h(s,v)(g_n(s,v)-1)\nu(dv)ds=\int_{\mathbb{X}_T}h(s,v)(g(s,v)-1)\nu(dv)ds;
   \end{eqnarray*}

b). Fix $N\in\mathbb{N}$. Given $\epsilon>0$, there exists a compact set $K_\epsilon\subset\mathbb{X}$, such that
    \begin{eqnarray*}
      \sup_{g\in S^N}\int_{[0,T]}\int_{K_\epsilon^c}|h(s,v)||g(s,v)-1|\nu(dv)ds\leq\epsilon.
    \end{eqnarray*}

c). For every $\eta>0$,
there exists $\delta>0$, we have
such that for any $A\in \mathcal{B}([0,T])$ satisfying $\lambda_T(A)<\delta$
\begin{eqnarray}\label{G_0H-inequation}
\sup_{g\in S^N}\int_A\int_{\mathbb{X}}h(s,v)|g(s,v)-1|\nu(dv)ds\leq \eta.
\end{eqnarray}
\end{lem}

Fix $N\in\mathbb{N}$. For any $\varphi_\epsilon\in \tilde{\mathbb{A}}^N$, consider the following controlled SPDEs
\
\begin{eqnarray}\label{eq SPDE 02}
d \widetilde{X}^\epsilon_t
&=&\mathcal{A}(t,\widetilde{X}^\epsilon_t)dt
+\int_{\mathbb{X}}f(t,\widetilde{X}^\epsilon_t,z)(\varphi_\epsilon(t,z)-1)\nu(dz)dt\nonumber\\
&&               +\epsilon\int_{\mathbb{X}}f(t,\widetilde{X}^\epsilon_{t-},z)\widetilde{N}^{\epsilon^{-1}\varphi_\epsilon}(dz,dt),
\end{eqnarray}
with initial condition $\widetilde{X}^\epsilon_0=x$.

Recall $\tilde{\mathbb{A}}^N$ in Theorem \ref{LDP-main-01}. Let $\vartheta_{\epsilon}=\frac{1}{\varphi_{\epsilon}}$. The following
lemma follows from Lemma 2.3 and Section 5.2 in \cite{Budhiraja-Dupuis-Maroulas.}. Recall the notations in Section \ref{Section Representation}, we have
\begin{lem}
\begin{eqnarray*}
\mathcal{E}^{\epsilon}_{t}(\vartheta_{\epsilon})&:=&\exp\Big\{\int_{(0,t]\times \mathbb{X}\times
[0,\epsilon^{-1}\varphi_{\epsilon}]}\log(\vartheta_{\epsilon}(s,x))
\bar{N}(\;ds\;dx\;dr)\\
&&\qquad +\int_{(0,t]\times \mathbb{X}\times
[0,\epsilon^{-1}\varphi_{\epsilon}]}(-\vartheta_{\epsilon}(s,x)+1)\bar{\nu}_{T}(\;ds\;dx\;dr)\Big\}
\end{eqnarray*}
Consequently,
$$
\mathbb{Q}^{\epsilon}_{t}(G)=\int_{G}\mathcal{E}^{\epsilon}_{t}(\vartheta_{\epsilon})\;d\bar{\mathbb{P}},
\quad
\text{for}\quad G\in \mathcal{B}(\bar{\mathbb{M}})
$$
defines a probability measure on $\bar{\mathbb{M}}$.
\end{lem}

By the fact that $\epsilon
N^{\epsilon^{-1}\varphi_{\epsilon}}$ under $\mathbb{Q}^{\epsilon}_{T}$ has the same law as
that of $\epsilon N^{\epsilon^{-1}}$ under $\bar{\mathbb{P}}$. From Theorem \ref{thm solution 01}, we see that
there exists a unique solution $\widetilde{X}^{\epsilon}$ to the controlled SPDE (\ref{eq SPDE 02}) which  satisfies (2) in Theorem \ref{thm solution 01}.

 By the definition of $\mathcal{G}^{\epsilon}$, we have
\begin{eqnarray}\label{define G-epsilon}
\widetilde{X}^{\epsilon}=\mathcal{G}^{\epsilon}\Big(\epsilon
N^{\epsilon^{-1}\varphi_{\epsilon}}\Big)
.
\end{eqnarray}

The following estimates (Lemmas \ref{lemma 4.1} and \ref{lemma 02}) will be useful.

\begin{lem}\label{lemma 4.1}
For $p=2,\ 2+\beta\ \text{or}\ \Upsilon$ in (H5), there exists $\epsilon_p,\ C_p>0$ such that
\begin{eqnarray*}
\sup_{\epsilon\in(0,\epsilon_p]}\mathbb{E}\Big(\sup_{t\in[0,T]}
\|\widetilde{X}^\epsilon_t\|^p_H\Big)
+\mathbb{E}\Big(\int_0^T\|\widetilde{X}^\epsilon_t\|^{p-2}_H
\|\widetilde{X}^\epsilon_t\|^{\alpha}_Vdt\Big)
\leq
C_{p}.
\end{eqnarray*}
\end{lem}
\begin{proof}
By ${\rm It\hat{o}'s}$ formula, we have
\begin{equation}\label{eq Xp}
\|\widetilde{X}^\epsilon_t\|^p_H
=\|x\|^p_H+I_1(t)+I_2(t)+I_3(t)+I_4(t),
\end{equation}
where
\[I_1(t)=\frac{p}{2}\int_0^t\|\widetilde{X}^\epsilon_s\|^{p-2}_H\Big(2\langle \mathcal{A}(s,\widetilde{X}^\epsilon_s), \widetilde{X}^\epsilon_s\rangle_{V^*,V}\Big)ds,\]
\[I_2(t)=   \int_0^t\int_{\mathbb{X}}
      p\|\widetilde{X}^\epsilon_{s-}\|^{p-2}_H\langle \epsilon f(s,\widetilde{X}^\epsilon_{s-},z), \widetilde{X}^\epsilon_{s-}\rangle_{H,H}\widetilde{N}^{\epsilon^{-1}
      \varphi_\epsilon}(dz,ds),\]
\begin{eqnarray*}
I_4(t)&=&
      \int_0^t\int_{\mathbb{X}}
          \Big[
            \|\widetilde{X}^\epsilon_{s-}+\epsilon f(s,\widetilde{X}^\epsilon_{s-},z)\|^p_{H}-\|\widetilde{X}^\epsilon_{s-}\|^p_{H}
            \\
            &&\ \ \ \ \ \ \ \ \ \ \ -
            p\|\widetilde{X}^\epsilon_{s-}\|^{p-2}_{H}\langle \epsilon f(s,\widetilde{X}^\epsilon_{s-},z), \widetilde{X}^\epsilon_{s-}\rangle_{H,H}
            \Big]
            N^{\epsilon^{-1}\varphi_\epsilon}(dz,ds),
            \end{eqnarray*}
and
\[I_4(t)=            p\int_0^t \|\widetilde{X}^\epsilon_s\|^{p-2}_{H}\langle \int_{\mathbb{X}}f(s,\widetilde{X}^\epsilon_s,z)(\varphi_\epsilon(s,z)-1),
\widetilde{X}^\epsilon_s\rangle_{H,H}\nu(dz)ds.\]

Note that by (H3),
\begin{eqnarray}\label{eq lemma2 I1}
  I_1(t)
&\leq&
   \frac{p}{2}\int_0^t\|\widetilde{X}^\epsilon_s\|^{p-2}_H\Big( F_s+F_s\|\widetilde{X}^\epsilon_s\|^2_H-\theta \|\widetilde{X}^\epsilon_s\|^\alpha_V\Big)ds\nonumber\\
&\leq&
   -\frac{\theta p}{2}\int_0^t\|\widetilde{X}^\epsilon_s\|^{p-2}_H\|\widetilde{X}^\epsilon_s\|^\alpha_Vds
      +
   \frac{p}{2}\int_0^t\Big[\Big(\|\widetilde{X}^\epsilon_s\|^{p}_H+1\Big)F_s+F_s\|\widetilde{X}^\epsilon_s\|^{p}_H\Big]ds\nonumber\\
&\leq&
   -\frac{\theta p}{2}\int_0^t\|\widetilde{X}^\epsilon_s\|^{p-2}_H\|\widetilde{X}^\epsilon_s\|^\alpha_Vds
      +
    \frac{p}{2}\int_0^tF_sds+\int_0^t pF_s\|\widetilde{X}^\epsilon_s\|^{p}_H ds,
\end{eqnarray}
and by (H5),
\begin{eqnarray}\label{eq lemma2 I4}
  I_4(t)
&\leq&
  p\int_0^t \|\widetilde{X}^\epsilon_s\|^{p-1}_{H} \int_{\mathbb{X}}\|f(s,\widetilde{X}^\epsilon_s,z)\|_H|(\varphi_\epsilon(s,z)-1)|\nu(dz)ds\nonumber\\
&\leq&
  p\int_0^t \|\widetilde{X}^\epsilon_s\|^{p-1}_{H}(1+\|\widetilde{X}^\epsilon_s\|_{H}) \int_{\mathbb{X}}L_f(s,z)|(\varphi_\epsilon(s,z)-1)|\nu(dz)ds\nonumber\\
&\leq&
  p\int_0^t\int_{\mathbb{X}}L_f(s,z)|(\varphi_\epsilon(s,z)-1)|\nu(dz)ds\nonumber\\
    &&+
  2p\int_0^t \|\widetilde{X}^\epsilon_s\|^{p}_{H}\int_{\mathbb{X}}L_f(s,z)|(\varphi_\epsilon(s,z)-1)|\nu(dz)ds.
\end{eqnarray}
By Gronwall's inequality, combining (\ref{eq Xp}) (\ref{eq lemma2 I1}), (\ref{eq lemma2 I4}) and Lemma \ref{Lemma-Condition-0,H-1,H},
\begin{eqnarray}\label{eq 2}
     &&\|\widetilde{X}^\epsilon_t\|^{p}_{H}+\frac{\theta p}{2}\int_0^t\|\widetilde{X}^\epsilon_s\|^{p-2}_H
     \|\widetilde{X}^\epsilon_s\|^\alpha_Vds\\
&\leq&\exp\Big(p\int_0^TF_sds+2pC_{L_f,N}\Big)\times
     \Big(
        \|x\|^p_H+\frac{p}{2}\int_0^TF_sds
          +
        \sup_{s\in[0,t]}|I_2(s)|+ pC_{L_f,N}\nonumber\\
        && \ \ \  +
        \int_0^t\int_{\mathbb{X}}c_p
            \Big(
               \|\widetilde{X}^\epsilon_{s-}\|^{p-2}_{H}\|\epsilon f(s,\widetilde{X}^\epsilon_{s-},z)\|^2_H + \|\epsilon f(s,\widetilde{X}^\epsilon_{s-},z)\|^p_H
            \Big)
        N^{\epsilon^{-1}\varphi_\epsilon}(dz,ds)
             \Big),\nonumber
\end{eqnarray}
we have used (4.9) in \cite{B-Liu-Zhu} to $I_3$, i.e.
$$
\Big|\|x+h\|_H^p-\|x\|_H^p-p\|x\|^{p-2}_H\langle x,h\rangle_{H,H}\Big|\leq c_p\Big(\|x\|^{p-2}_H\|h\|^2_H+\|h\|^p_H\Big),
\ \ \forall x,h\in H.
$$

By Lemma \ref{Lemma-Condition-0,H-1,H}, we have
\begin{eqnarray}\label{eq 2-01}
 &&  \mathbb{E}\Big(\sup_{s\in[0,T]}|I_2(s)|\Big)\nonumber\\
&\leq&
    \mathbb{E}\Big( \int_0^T\int_{\mathbb{X}}
       \epsilon^2 p^2\|\widetilde{X}^\epsilon({s-})\|^{2p-4}_H\langle f(s,\widetilde{X}^\epsilon_{s-},z), \widetilde{X}^\epsilon({s-})\rangle^2_{H,H}N^{\epsilon^{-1}\varphi_\epsilon}(dz,ds)\Big)^{1/2}\nonumber\\
&\leq&
   \mathbb{E}\Big( \int_0^T\int_{\mathbb{X}}
        \epsilon^2 p^2\|\widetilde{X}^\epsilon({s-})\|^{2p-2}_H L^2_f(s,z)\Big(\|\widetilde{X}^\epsilon_{{s-}}\|_H+1\Big)^2
                   N^{\epsilon^{-1}\varphi_\epsilon}(dz,ds)
             \Big)^{1/2}\nonumber\\
&\leq&
  \mathbb{E}\Big( \sup_{s\in[0,T]}\|\widetilde{X}^\epsilon_s\|^{p}_H
           \cdot
         \epsilon^2 p^2\int_0^T\int_{\mathbb{X}}
                  \|\widetilde{X}^\epsilon_{s-}\|^{p-2}_H L^2_f(s,z)\Big(\|\widetilde{X}^\epsilon_{s-}\|_H+1\Big)^2
                   N^{\epsilon^{-1}\varphi_\epsilon}(dz,ds)
             \Big)^{1/2}\nonumber\\
&\leq&
  \frac{1}{4}\mathbb{E}\Big( \sup_{s\in[0,T]}\|\widetilde{X}^\epsilon_s\|^{p}_H\Big)\nonumber\\
&&           +
         16\epsilon p^2\mathbb{E}\Big[\Big(\sup_{s\in[0,T]}\|\widetilde{X}^\epsilon_s\|^{p}_H+1\Big)
                   \int_0^T\int_{\mathbb{X}}
                  L^2_f(s,z)\varphi_\epsilon(s,z)\nu(dz)ds
             \Big]\nonumber\\
&\leq&
     (\frac{1}{4}+16\epsilon p^2 C_{L_f,2,2,N})\mathbb{E}\Big( \sup_{s\in[0,T]}\|\widetilde{X}^\epsilon_s\|^{p}_H\Big)
        +
      16\epsilon p^2 C_{L_f,2,2,N}.
\end{eqnarray}

On the other hand, by Lemma \ref{Lemma-Condition-0,H-1,H} again, we have
\begin{eqnarray}\label{eq 2-02}
&&    \mathbb{E}\Big(\int_0^T\int_{\mathbb{X}}c_p
                          \|\widetilde{X}^\epsilon_s\|^{p-2}_{H}\|\epsilon f(s,\widetilde{X}^\epsilon_s,z)\|^2_H
                         N^{\epsilon^{-1}\varphi_\epsilon}(dz,ds)\Big)\nonumber\\
&\leq&
    \epsilon c_p\mathbb{E}\Big(\int_0^T\int_{\mathbb{X}}
                          \|\widetilde{X}^\epsilon_s\|^{p-2}_{H}L^2_f(s,z)(\|\widetilde{X}^\epsilon_s\|_H+1)^2\varphi_\epsilon(s,z)\nu(dz)ds
                         \Big)\nonumber\\
&\leq&
    \epsilon c_p \mathbb{E}\Big[\Big(\sup_{s\in[0,T]}\|\widetilde{X}^\epsilon_s\|^{p}_H+1\Big)\int_0^T\int_{\mathbb{X}}
                             L^2_f(s,z)\varphi_\epsilon(s,z)\nu(dz)ds\Big]\nonumber\\
&\leq&
    \epsilon c_p C_{L_f,p,p,N}\mathbb{E}\Big(\sup_{s\in[0,T]}\|\widetilde{X}^\epsilon_s\|^{p}_H\Big)
    +
    \epsilon c_p C_{L_f,2,2,N},
\end{eqnarray}
and
\begin{eqnarray}\label{eq 2-03}
&&\mathbb{E}\Big(\int_0^T\int_{\mathbb{X}}c_p
               \|\epsilon f(s,\widetilde{X}^\epsilon_s,z)\|^p_H
        N^{\epsilon^{-1}\varphi_\epsilon}(dz,ds)
         \Big)\nonumber\\
&=&
  \epsilon^{p-1} c_p \mathbb{E}\Big(\int_0^T\int_{\mathbb{X}} \|f(s,\widetilde{X}^\epsilon_s,z)\|^p_H \varphi_\epsilon(s,z)\nu(dz)ds \Big)\nonumber\\
&\leq&
  \epsilon^{p-1} c_p \mathbb{E}\Big[\Big(\sup_{s\in[0,T]}\|\widetilde{X}^\epsilon_s\|^{p}_H+1\Big)\int_0^T\int_{\mathbb{X}}
                             L^p_f(s,z)\varphi_\epsilon(s,z)\nu(dz)ds\Big]\nonumber\\
&\leq&
    \epsilon^{p-1} c_p C_{L_f,p,p,N} \mathbb{E}\Big(\sup_{s\in[0,T]}\|\widetilde{X}^\epsilon_s\|^{p}_H\Big)
    +
    \epsilon^{p-1} c_p C_{L_f,p,p,N}.
\end{eqnarray}

Combining (\ref{eq 2})--(\ref{eq 2-03}), we obtain that there exists $\epsilon_p>0$ such that
\begin{eqnarray*}
  \sup_{\epsilon\in(0,\epsilon_p]}
       \Big[
    \mathbb{E}\Big(\sup_{s\in[0,T]}\|\widetilde{X}^\epsilon_s\|^{p}_H\Big)
    +
    \frac{\theta p}{2}\mathbb{E}\Big(\int_0^T\|\widetilde{X}^\epsilon_s\|^{p-2}_H\|\widetilde{X}^\epsilon_s\|^\alpha_Vds\Big)
       \Big]
 \leq
   C_{N,p,T,\|x\|_H,\int_0^TF_sds,L_f}.
\end{eqnarray*}
The proof is complete.
\end{proof}

\begin{lem}\label{lemma 02}
For $p=\frac{\Upsilon}{2}$, there exist $C_p$ such that
\begin{eqnarray*}
\sup_{\epsilon\in(0,\epsilon_{2p}]}\mathbb{E}
\Big(\int_0^T\|\widetilde{X}^\epsilon_s\|^\alpha_V ds\Big)^p
\leq
C_p.
\end{eqnarray*}
Here $\epsilon_{2p}$ comes from Lemma \ref{lemma 4.1}.

\end{lem}
\begin{proof}
Consider $p=2$ in (\ref{eq 2}), we have
\begin{equation}\label{eq lemma2 01}\theta \int_0^t\|\widetilde{X}^\epsilon_s\|^\alpha_Vds
\le C_{N,T,\int_0^TF_sds,L_f}\(\|x\|^2_H+\int_0^TF_sds
          +
        \sup_{s\in[0,t]}|I_2(s)|+2C_{L_f,N}+J(t)\),
        \end{equation}
where
\[J(t)=\int_0^t\int_{\mathbb{X}}c_2
            \Big(
                \|\epsilon f(s,\widetilde{X}^\epsilon_{s-},z)\|^2_H
            \Big)
        N^{\epsilon^{-1}\varphi_\epsilon}(dz,ds).\]
 In the following calculations, we take $p=\frac{\Upsilon}{2}$. Note that
 \begin{eqnarray*}
 \EE\(|J(t)|^p\)
 &\le &c_p \mathbb{E}\Big(\Big|\int_0^T\int_{\mathbb{X}}
            \Big(
                \|\epsilon f(s,\widetilde{X}^\epsilon_{s-},z)\|^2_H
            \Big)
        \widetilde{N}^{\epsilon^{-1}\varphi_\epsilon}(dz,ds)\Big|^p
         \Big)\nonumber\\
&&      +
    c_p \mathbb{E}\Big(\Big|\int_0^T\int_{\mathbb{X}}
            \Big(
                \epsilon\| f(s,\widetilde{X}^\epsilon_s,z)\|^2_H
            \Big)\varphi_\epsilon(s,z)\nu(dz)ds
        \Big|^p
         \Big).
\end{eqnarray*}
 By Kunita's first inequality (refer to Theorem 4.4.23 in \cite{David}),
  we can continue with
 \begin{eqnarray*}
 \EE\(|J(t)|^p\)
 &\le &c_p \epsilon^{2p-1}\mathbb{E}\Big(\int_0^T\int_{\mathbb{X}}\| f(s,\widetilde{X}^\epsilon_s,z)\|^{2p}_H
            \varphi_\epsilon(s,z)\nu(dz)ds\Big)\nonumber\\
&&                        +
            c_p\epsilon^{3p/2}\mathbb{E}\Big(\int_0^T\int_{\mathbb{X}}\| f(s,\widetilde{X}^\epsilon_s,z)\|^{4}_H
            \varphi_\epsilon(s,z)\nu(dz)ds\Big)^{p/2}\nonumber\\
            &&+
            c_p\epsilon^{p}\mathbb{E}\Big(\int_0^T\int_{\mathbb{X}}\| f(s,\widetilde{X}^\epsilon_s,z)\|^{2}_H
            \Big)\varphi_\epsilon(s,z)\nu(dz)ds\Big)^{p}.
\end{eqnarray*}
Thus, by Lemma \ref{Lemma-Condition-0,H-1,H}, we have
\begin{eqnarray}\label{eq lemma2 02}
&&\EE\(|J(t)|^p\)\\
&\leq&
c_p\mathbb{E}\Big(1+\sup_{s\in[0,T]}\|\widetilde{X}^\epsilon_s\|_H\Big)^{2p}
\Big(
  \epsilon^{2p-1}\sup_{\varphi\in S^N}\int_0^T\int_{\mathbb{X}}L_{f}^{2p}(s,z)\varphi(s,z)\nu(dz)ds\nonumber\\
&& \ \ \ \ \ \ \ \ \ \ \ \ \ \ \ \ \ \ \ \ \ \ \ \ \ \ \ \ \ \ \ \ \  +
  \epsilon^{3p/2}\Big(\sup_{\varphi\in S^N}\int_0^T\int_{\mathbb{X}}L_{f}^4(s,z)\varphi(s,z)\nu(dz)ds\Big)^{p/2}\nonumber\\
  && \ \ \ \ \ \ \ \ \ \ \ \ \ \ \ \ \ \ \ \ \ \ \ \ \ \ \ \ \ \ \ \ \  +
    \epsilon^{p}\Big(\sup_{\varphi\in S^N}\int_0^T\int_{\mathbb{X}}L_{f}^2(s,z)\varphi(s,z)\nu(dz)ds\Big)^{p}
\Big)\nonumber\\
&\leq&
c_p\mathbb{E}\Big(1+\sup_{s\in[0,T]}\|\widetilde{X}^\epsilon_s\|\Big)^{2p}
\Big(
    \epsilon^{2p-1}C_{L_f,2p,2p,N}+\epsilon^{3p/2}\Big(C_{L_f,4,4,N}\Big)^{p/2}+\epsilon^{p}\Big(C_{L_f,2,2,N}\Big)^{p}
\Big).\nonumber
\end{eqnarray}
By Kunita's first inequality again,
\begin{eqnarray}\label{eq lemma2 03}
  &&\mathbb{E}\(\sup_{s\in[0,T]}|I_2(s)|^p\)\nonumber\\
&\leq&
 c_p\epsilon^{p-1}\mathbb{E}\Big(\int_0^T\int_{\mathbb{X}}\Big|\langle f(s,\widetilde{X}^\epsilon_{s},z), \widetilde{X}^\epsilon_s\rangle_{H,H}\Big|^p\varphi_\epsilon(s,z)\nu(dz)ds\Big)\nonumber\\
 &&
 +
 c_p\epsilon^{p/2}\mathbb{E}\Big(\int_0^T\int_{\mathbb{X}}\Big|\langle f(s,\widetilde{X}^\epsilon_{s},z), \widetilde{X}^\epsilon_s\rangle_{H,H}\Big|^2\varphi_\epsilon(s,z)\nu(dz)ds\Big)^{p/2}\nonumber\\
&\leq&
  c_p\epsilon^{p-1}\mathbb{E}\Big(\int_0^T\int_{\mathbb{X}}\|\widetilde{X}^\epsilon_{s}\|^p_HL_f^p(s,z)\Big(1+\|\widetilde{X}^\epsilon_{s}\|_H\Big)^p
          \varphi_\epsilon(s,z)\nu(dz)ds\Big)\nonumber\\
 &&+
 c_p\epsilon^{p/2}\mathbb{E}\Big(\int_0^T\int_{\mathbb{X}}\|\widetilde{X}^\epsilon_{s}\|^2_HL_f^2(s,z)
  \Big(1+\|\widetilde{X}^\epsilon_{s}\|_H\Big)^2\varphi_\epsilon(s,z)\nu(dz)ds\Big)^{p/2}\nonumber\\
&\leq&
  c_p\epsilon^{p-1}\mathbb{E}\Big(1+\sup_{s\in[0,T]}\|\widetilde{X}^\epsilon_s\|_H\Big)^{2p}
   \sup_{\varphi\in S^N}\int_0^T\int_{\mathbb{X}}L_{f}^{p}(s,z)\varphi(s,z)\nu(dz)ds\nonumber\\
   &&+
   c_p\epsilon^{p/2}\mathbb{E}\Big(1+\sup_{s\in[0,T]}\|\widetilde{X}^\epsilon_s\|_H\Big)^{2p}
   \Big(\sup_{\varphi\in S^N}\int_0^T\int_{\mathbb{X}}L_{f}^{2}(s,z)\varphi(s,z)\nu(dz)ds\Big)^{p/2}\nonumber\\
&\leq&
 c_p\mathbb{E}\Big(1+\sup_{s\in[0,T]}\|\widetilde{X}^\epsilon_s\|\Big)^{2p}
\Big(
    \epsilon^{p-1}C_{L_f,p,p,N}+\epsilon^{p/2}\Big(C_{L_f,2,2,N}\Big)^{p/2}
\Big).
\end{eqnarray}
Lemma \ref{lemma 4.1} and (\ref{eq lemma2 01})--(\ref{eq lemma2 03}) imply this lemma.

\end{proof}

Finally, we prove the tightness of $\{\widetilde{X}^\epsilon\}$.

\begin{prop}\label{lem 4.6} For some $\epsilon_0>0$, $\{\widetilde{X}^\epsilon\}_{\epsilon\in(0,\epsilon_0]}$ is tight in $D([0,T], V^*)$ with the Skorohod topology.
Moreover, set
$$M^\epsilon_t=\int_0^t\int_{\mathbb{X}}
     \epsilon f(s,\widetilde{X}^\epsilon_{s-},z)\widetilde{N}^{\epsilon^{-1}\varphi_\epsilon}(dz,ds),$$
$$Z^\epsilon_t=\int_0^t\int_{\mathbb{X}}f(s,\widetilde{X}^\epsilon_s,z)(\varphi_\epsilon(s,z)-1)\nu(dz)ds,$$
$$Y^\epsilon_t=\int_0^t\mathcal{A}(s,\widetilde{X}^\epsilon_s)ds,$$
then
\begin{itemize}
 \item[(a)] $\lim_{\epsilon\rightarrow0}\mathbb{E}\Big(\sup_{t\in[0,T]}\Big\|M^\epsilon_t\Big\|^2_H\Big)=0$,

 \item[(b)] $(Z^\epsilon_t)_{0 \le t \le T}$ is tight in $C([0,T],V^*)$,

 \item[(c)] $(Y^\epsilon_t)_{0 \le t \le T}$ is tight in $C([0,T],V^*)$.
\end{itemize}
 \end{prop}
\begin{proof} (a).
By Lemma \ref{Lemma-Condition-0,H-1,H}, we have
\begin{eqnarray}\label{eq tihgt 01}
&&    \mathbb{E}\Big(\sup_{t\in[0,T]}\Big\|M^\epsilon_t\Big\|^2_H\Big)\nonumber\\
&\leq&
 C\epsilon \mathbb{E}\Big(\int_0^T\int_{\mathbb{X}}\| f(s,\widetilde{X}^\epsilon_s,z)\|^{2}_H
            \Big)\varphi_\epsilon(s,z)\nu(dz)ds\Big)\nonumber\\
&\leq&
 C\epsilon \mathbb{E}\Big(\int_0^T\int_{\mathbb{X}}L_f^2(s,z)\Big(1+\|\widetilde{X}^\epsilon_{s}\|_H\Big)^2
          \varphi_\epsilon(s,z)\nu(dz)ds\Big)\nonumber\\
&\leq&
C\epsilon \mathbb{E}\Big(1+\sup_{s\in[0,T]}\|\widetilde{X}^\epsilon_s\|_H\Big)^{2}
   \Big(\sup_{\varphi\in S^N}\int_0^T\int_{\mathbb{X}}L_{f}^{2}(s,z)\varphi(s,z)\nu(dz)ds\Big)\nonumber\\
&\leq&
    C\epsilon \mathbb{E}\Big(1+\sup_{s\in[0,T]}\|\widetilde{X}^\epsilon_s\|_H\Big)^{2}C_{L_f,2,2,N}\nonumber\\
&\rightarrow &0,\ \ \text{as}\ \epsilon\rightarrow 0.
\end{eqnarray}

(b).
It is sufficient to prove that for any $\delta>0$, there exists a compact subset $K_\delta\subset C([0,T],V^*)$ such that
$$
\mathbb{P}(Z^\epsilon\in K_\delta)>1-\delta.
$$

Denote
$$
\mathcal{D}_{M,N}=
   \Big\{
    (r_t,g_t):\ r_{\cdot}\in D([0,T],H)\cap L^\alpha([0,T],V),\ \sup_{t\in[0,T]}\|r_t\|_H\leq M;\ \ g\in S^N
   \Big\},
$$
$$
\mathcal{R}(\mathcal{D}_{M,N})=
     \Big\{
       y_{\cdot}=\int_0^{\cdot}\int_{\mathbb{X}}f(s,r_s,z)(g(s,z)-1)\nu(dz)ds,\ \ (r,g)\in \mathcal{D}_{M,N}
     \Big\}.
$$

For any $y\in \mathcal{R}(\mathcal{D}_{M,N})$, we have
\begin{eqnarray}\label{eq tight step 2 01}
  \|y_t-y_s\|_{H}
&\leq&
  \int_s^{t}\int_{\mathbb{X}}\|f(l,r(l),z)\|_H|g(l,z)-1|\nu(dz)dl\nonumber\\
&\leq&
   \sup_{l\in[s,t]}(1+\|r(l)\|_H)\int_s^{t}\int_{\mathbb{X}}L_f(l,z)|g(l,z)-1|\nu(dz)dl\nonumber\\
&\leq&
   (M+1)\sup_{\varphi\in S^N}\int_s^{t}\int_{\mathbb{X}}L_f(l,z)|\varphi(l,z)-1|\nu(dz)dl.
\end{eqnarray}
Applying Lemma \ref{Lemma-Condition-0,H-1,H}, c) in Lemma \ref{lem-thm2-02} and (\ref{eq tight step 2 01}), we obtain the following:

(1) for any $\eta>0$, there exists $\varpi>0$ (independent on $y$) such that for any $s,t\in[0,T]$ and $|t-s|\leq\varpi$
$$
  \|y_t-y_s\|_{H}\leq \eta,\ \ \forall y\in \mathcal{R}(\mathcal{D}_{M,N}),
$$

(2)
\begin{eqnarray*}
\sup_{y\in \mathcal{R}(\mathcal{D}_{M,N})}\sup_{t\in[0,T]}\|y_t\|_{H}
=\sup_{y\in \mathcal{R}(\mathcal{D}_{M,N})}\sup_{t\in[0,T]}\|y_t-y_0\|_{H}
\leq (M+1)C_{L_f,N}.
\end{eqnarray*}

Since $V\hookrightarrow H$ is compact, we also have $H\hookrightarrow V^*$ compactly. By Ascoli-Arzel\'a's theorem, the complement of
$\mathcal{R}(\mathcal{D}_{M,N})$ in $C([0,T],V^*)$, denoted by $\overline{\mathcal{R}}(\mathcal{D}_{M,N})$, is a compact subset in $C([0,T],V^*)$.

On the other hand,
\begin{eqnarray*}
 \mathbb{P}(Z^\epsilon\in \overline{\mathcal{R}}(\mathcal{D}_{M,N}))
 &\geq&
 \mathbb{P}(\sup_{t\in[0,T]}\|\widetilde{X}^\epsilon_t\|_H\leq M)\\
 &=&
 1-\mathbb{P}(\sup_{t\in[0,T]}\|\widetilde{X}^\epsilon_t\|_H> M)\\
 &\geq&
 1-\mathbb{E}(\sup_{t\in[0,T]}\|\widetilde{X}^\epsilon_t\|^2_H)/M^2\\
 &\geq&
 1-C_{2}/M^2,
\end{eqnarray*}
we have applied Lemma \ref{lemma 4.1} in the last inequality and this establishes that $\{Z^\epsilon\}$ is tight in $C([0,T],V^*)$.

\vskip 0.3cm
(c). By Lemmas \ref{lemma 4.1} and \ref{lemma 02}, recall $\eta_0$ in (H5), let $p=\alpha+\eta_0$,
we have
\begin{eqnarray*}
 \mathbb{E}\|Y^\epsilon_t-Y^\epsilon_s\|^p_{V^*}
 &\leq&
 \mathbb{E}\Big|\int_s^t\|\mathcal{A}(l,\widetilde{X}^\epsilon_l)\|_{V^*}dl\Big|^p\nonumber\\
 &\leq&
 |t-s|^{p/\alpha}\mathbb{E}\Big(\int_s^t\|\mathcal{A}(l,\widetilde{X}^\epsilon_l)\|^{\frac{\alpha}{\alpha-1}}_{V^*}dl\Big)^{\frac{(\alpha-1)p}{\alpha}}\nonumber\\
 &\leq&
 |t-s|^{p/\alpha}\mathbb{E}\Big(\int_s^t(F_l+C\|\widetilde{X}^\epsilon_l\|^\alpha_V)(1+\|\widetilde{X}^\epsilon_l\|_H^\beta) dl\Big)^{\frac{(\alpha-1)p}{\alpha}}\nonumber\\
 &\leq&
 |t-s|^{p/\alpha}\Big[\mathbb{E}\Big(\sup_{l\in[0,T]}(1+\|\widetilde{X}^\epsilon_l\|_H^\beta)^{\frac{2(\alpha-1)p}{\alpha}}\Big)\\
&&   \ \ \ \ \ \ \ \ \ \ \ \ \ \       +
         \mathbb{E}\Big(\int_s^tF_l+C\|\widetilde{X}^\epsilon_l\|^\alpha_Vdl\Big)^{\frac{2(\alpha-1)p}{\alpha}}
        \Big]\nonumber\\
 &\leq&
 C_{\alpha,p,F}|t-s|^{p/\alpha}.
\end{eqnarray*}

Hence, a direct application of Kolmogorov's criterion, for every $\varpi\in(0,\frac{1}{\alpha}-\frac{1}{p})$, there exists constant $C_\varpi$ independent on $\epsilon$ such that
\begin{eqnarray}\label{eq tight step 3 01}
 \mathbb{E}\Big(\sup_{t\neq s\in[0,T]}\frac{\|Y^\epsilon_t-Y^\epsilon_s\|^p_{V^*}}{|t-s|^{p\varpi}}\Big)\leq C_\varpi.
\end{eqnarray}

On the other hand, by (\ref{eq SPDE 02}), we have
\begin{eqnarray*}
\widetilde{X}^\epsilon_t=x+Y^\epsilon_t+Z^\epsilon_t
      +M^\epsilon_t.
\end{eqnarray*}
Then
\begin{eqnarray}\label{eq step 3 01}
 &&\mathbb{E}\Big(\sup_{t\in[0,T]}\|Y^\epsilon_t\|^2_H\Big)\\
 &\leq&
 C\Big[ \|x\|^2_H+\mathbb{E}\Big(\sup_{t\in[0,T]}\|\widetilde{X}^\epsilon_t\|^2_H\Big)
        +
        \mathbb{E}\Big(\sup_{t\in[0,T]}\|Z^\epsilon_t\|^2_H\Big)
        +
        \mathbb{E}\Big(\sup_{t\in[0,T]}\|M^\epsilon_t\|^2_H\Big)
  \Big].\nonumber
\end{eqnarray}
Notice that
\begin{eqnarray}\label{eq step 3 02}
  &&\mathbb{E}\Big(\sup_{t\in[0,T]}\|Z^\epsilon_t\|^2_H\Big)\nonumber\\
  &\leq&
  \mathbb{E}\Big(\int_0^T\int_{\mathbb{X}}\|f(s,\widetilde{X}^\epsilon_s,z)\|_H|\varphi_\epsilon(s,z)-1|\nu(dz)ds\Big)^2\nonumber\\
  &\leq&
  C\mathbb{E}\Big(1+\sup_{t\in[0,T]}\|\widetilde{X}^\epsilon_t\|_H\Big)^2
  \Big(\sup_{\varphi\in S^N}\int_0^T\int_{\mathbb{X}}L_f(s,z)|\varphi(s,z)-1|\nu(dz)ds\Big)^2\nonumber\\
  &\leq&
  CC^2_{L_f,N}\mathbb{E}\Big(1+\sup_{t\in[0,T]}\|\widetilde{X}^\epsilon_t\|_H\Big)^2
\end{eqnarray}
By Lemma \ref{lemma 4.1}, (\ref{eq step 3 01}), (\ref{eq step 3 02}) and  (\ref{eq tihgt 01}), we have
\begin{eqnarray}\label{eq step 3}
\mathbb{E}\Big(\sup_{t\in[0,T]}\|Y^\epsilon_t\|^2_H\Big)\leq C<\infty,
\end{eqnarray}
where $C$ is independent of $\epsilon$.

For $\varpi\in(0,1)$ and $R>0$. Set
$$
K_{R,\varpi}:=\Big\{j\in C([0,T],V^*):\ \sup_{t\in[0,T]}\|j_t\|_H+\sup_{s\neq t\in[0,T]}\frac{\|j_t-j_s\|_{V^*}}{|t-s|^\varpi}\leq R\Big\}.
$$
Since $V\hookrightarrow H$ is compact, we also have $H\hookrightarrow V^*$ compactly. By Ascoli-Arzel\'a's theorem,
$K_{R,\varpi}$ is a compact subset of $C([0,T],V^*)$. By (\ref{eq tight step 3 01}), (\ref{eq step 3}) and Chebyschev's inequality,
 for some $\varpi\in(0,1)$ and any $R>0$, we have
$$
\mathbb{P}\Big(Y^\epsilon\not\in K_{R,\varpi}\Big)\geq\frac{C_{T,\varpi}}{R}.
$$
This implies the tightness of $\{Y^\epsilon\}$  in $C([0,T],V^*)$.

The tightness of $\{\widetilde{X}^\epsilon\}$ in $D([0,T],V^*)$ then follows from (\ref{eq SPDE 02}) and the conclusions proved above.
\end{proof}

\section{Convergency of the processes}
\setcounter{equation}{0}
\renewcommand{\theequation}{\thesection.\arabic{equation}}

With the tightness result obtained in the last section, we now characterize the limit points and derive limiting results for the processes.

Throughout this section, we assume that for almost all $\omega$, as $\epsilon\to 0$, $\varphi_\epsilon(\cdot,\cdot)(\omega)$ converges to $\varphi(\cdot,\cdot)(\omega)$ in
$S^N$ weakly, and $X^\epsilon(\omega)$ converges to $X(\omega)$ in $D([0,T],V^*)$ strongly  with supremum norm.

Set
$$
\mathcal{K}=L^\alpha([0,T]\times\Omega\rightarrow V;dt\times \bar{\mathbb{P}}),
$$
$$
\mathcal{K}^*=L^\frac{\alpha}{\alpha-1}([0,T]\times\Omega\rightarrow V^*;dt\times \bar{\mathbb{P}}).
$$

\begin{lem}\label{lem0517a1}
There exists a subsequence $(\epsilon_k)$, $\bar{X}\in \mathcal{K}\cap L^\infty([0,T],L^{\beta+2}(\Omega,H))$ and
$Y\in \mathcal{K}^*$ such that

(i) $X^{\epsilon_k}\rightarrow \bar{X}$  in $\mathcal{K}$ weakly and  in $L^\infty([0,T],L^{\beta+2}(\Omega,H))$ in weak-star topology,

(ii) $\mathcal{A}(\cdot,X^{\epsilon_k})\rightarrow Y$  in $\mathcal{K}^*$ weakly,

(iii) \begin{eqnarray*}
\lim_{\epsilon\rightarrow0}\mathbb{E}\Big(\sup_{t\in[0,T]}\|X^\epsilon_t-X_t\|_{V^*}\Big)=0,
\end{eqnarray*}
and for $m=\frac{\alpha}{\alpha+1}$,
\begin{eqnarray*}
\lim_{\epsilon\rightarrow0}\mathbb{E}\int_0^T\|X^\epsilon_t-X_t\|^{2m}_Hdt=0.
\end{eqnarray*}
\end{lem}

\begin{proof}
(i) following from Lemma \ref{lemma 4.1}. For (ii),
by Lemma \ref{lemma 4.1} again,
\begin{eqnarray}
\|\mathcal{A}(\cdot,X^\epsilon(\cdot))\|_{\mathcal{K}^*}^{\frac{\alpha-1}{\alpha}}
&=&
\mathbb{E}\Big(\int_0^T\|\mathcal{A}(t,X^\epsilon_t)\|_{V^*}^{\frac{\alpha}{\alpha-1}}dt\Big)\nonumber\\
&\leq&
\mathbb{E}\Big(\int_0^T(F_t+C\|X^\epsilon_t\|_V^\alpha)(1+\|X^\epsilon_t\|^\beta_H)dt\Big)\nonumber\\
&\leq&
C<\infty.
\end{eqnarray}

Lemma \ref{lemma 4.1} implies
\begin{eqnarray}\label{eq lemma3 01}
 \mathbb{E}\Big(\sup_{t\in[0,T]}\|X^\epsilon_t\|^2_H\Big)\leq C_{2,N,x},
\end{eqnarray}
and
\begin{eqnarray}\label{eq lemma3 02}
 \mathbb{E}\Big(\int_0^T\|X^\epsilon_t\|^{\alpha}_Vdt\Big)\leq C.
\end{eqnarray}
Hence, by the strong convergence of $X^\epsilon(\omega)$ to $X(\omega)$ in $D([0,T],V^*)$ with sup norm, Fatou's lemma, (\ref{eq lemma3 01}) and (\ref{eq lemma3 02}), we have
\begin{eqnarray}\label{eq lemma3 03}
 \mathbb{E}\Big(\sup_{t\in[0,T]}\|X_t\|^2_H\Big)\leq\liminf_{\epsilon\rightarrow0}\mathbb{E}\Big(\sup_{t\in[0,T]}\|X^\epsilon_t\|^2_H\Big)\leq C_{2,N,x},
\end{eqnarray}
\begin{eqnarray}\label{eq lemma3 04}
 \mathbb{E}\Big(\int_0^T\|X_t\|^{\alpha}_Vdt\Big)\leq\liminf_{\epsilon\rightarrow0}\mathbb{E}\Big(\int_0^T\|X^\epsilon_t\|^{\alpha}_Vdt\Big)\leq C.
\end{eqnarray}
and
\begin{eqnarray}\label{eq lemma3 05}
\lim_{\epsilon\rightarrow0}\mathbb{E}\Big(\sup_{t\in[0,T]}\|X^\epsilon_t-X_t\|_{V^*}\Big)=0.
\end{eqnarray}
(\ref{eq lemma3 05}) can be seen as following. Set
$$
\Omega^\epsilon_\delta=\{\omega: \sup_{t\in[0,T]}\|X^\epsilon_t-X_t\|_{V^*}\geq\delta\}.
$$
The strong convergence of $X^\epsilon(\omega)$ to $X(\omega)$ in $D([0,T],V^*)$ with sup norm implies
\begin{eqnarray}\label{eq lemma3 05 01}
\lim_{\epsilon\rightarrow 0}\mathbb{P}(\Omega^\epsilon_\delta)=0,\ \ \forall \delta>0.
\end{eqnarray}
Applying (\ref{eq lemma3 05 01}), (\ref{eq lemma3 01}) and (\ref{eq lemma3 03}) to (\ref{eq lemma3 05}), we have
\begin{eqnarray*}
&&   \lim_{\epsilon\rightarrow0}\mathbb{E}\Big(\sup_{t\in[0,T]}\|X^\epsilon_t-X_t\|_{V^*}\Big)\\
&=&
     \lim_{\epsilon\rightarrow0}
          \Big[
             \mathbb{E}\Big(\sup_{t\in[0,T]}\|X^\epsilon_t-X_t\|_{V^*}\cdot 1_{\Omega^\epsilon_\delta}\Big)
                 +
             \mathbb{E}\Big(\sup_{t\in[0,T]}\|X^\epsilon_t-X_t\|_{V^*}\cdot 1_{(\Omega^\epsilon_\delta)^c}\Big)
          \Big]\\
&\leq&
    \delta
    +
    \lim_{\epsilon\rightarrow0} \Big(\mathbb{E}\Big(\sup_{t\in[0,T]}\|X^\epsilon_t-X_t\|^2_{V^*}\Big)\Big)^{1/2}
                              \cdot
                                \Big(\mathbb{P}(\Omega^\epsilon_\delta)\Big)^{1/2}\\
&\leq&
\delta.
\end{eqnarray*}
The arbitrary of $\delta$ implies (\ref{eq lemma3 05}).

Taking $m=\frac{\alpha}{\alpha+1}$, we get
\begin{eqnarray*}
\mathbb{E}\int_0^T\|X^\epsilon_t-X_t\|^{2m}_Hdt
&=&
\mathbb{E}\int_0^T\langle X^\epsilon_t-X_t,X^\epsilon_t-X_t\rangle_{V^*,V}^{m}dt\\
&\leq&
\mathbb{E}\int_0^T\| X^\epsilon_t-X_t\|^m_{V^*}\|X^\epsilon_t-X_t\|_V^{m}dt\\
&\leq&
\Big(\mathbb{E}\int_0^T\| X^\epsilon_t-X_t\|_{V^*}dt\Big)^{\frac{\alpha-m}{\alpha}}
\Big(\mathbb{E}\int_0^T\|X^\epsilon_t-X_t\|^{\alpha}_Vdt\Big)^{\frac{m}{\alpha}}.
\end{eqnarray*}
Combining (\ref{eq lemma3 02}), (\ref{eq lemma3 04}) and (\ref{eq lemma3 05}), we have
\begin{eqnarray}\label{eq lemma3 06}
\lim_{\epsilon\rightarrow0}\mathbb{E}\int_0^T\|X^\epsilon_t-X_t\|^{2m}_Hdt=0.
\end{eqnarray}

\end{proof}

\begin{lem}\label{lem0517a2}
For any $h\in H$, we have
\begin{eqnarray}\label{eq star1}
&&\lim_{\epsilon_k\rightarrow0}\langle \int_0^t\int_{\mathbb{X}}f(s,X^{\epsilon_k}_s,z)
(\varphi_{\epsilon_k}(s,z)-1)\nu(dz)ds,h\rangle_{H,H}\nonumber\\
&=&
\langle \int_0^t\int_{\mathbb{X}}f(s,X_s,z)(\varphi(s,z)-1)\nu(dz)ds,h\rangle_{H,H}.
\end{eqnarray}
\end{lem}

\begin{proof}
Denote $\zeta(s,z)=\langle f(s,X_s,z),h\rangle_{H,H}$. Since $\sup_{s\in[0,T]}\|X_s\|_H<\infty,\;\; \mathbb{P}$-a.s., and $L_f\in\mathcal{H}_2$, it follows from Remark \ref{Remark 2} and Lemma \ref{lem-thm2-02} that
\begin{eqnarray}\label{eq star2}
&& \lim_{\epsilon_k\rightarrow0}\langle \int_0^t\int_{\mathbb{X}}f(s,X_s,z)(\varphi_{\epsilon_k}(s,z)-1)
\nu(dz)ds,h\rangle_{H,H}\nonumber\\
&=&
\langle \int_0^t\int_{\mathbb{X}}f(s,X_s,z)(\varphi(s,z)-1)\nu(dz)ds,h\rangle_{H,H}.
\end{eqnarray}

For any $\delta>0$, denote $A_{\delta,\epsilon}(\omega):=\Big\{s\in[0,T]:\ \|X^\epsilon_s-X_s\|_H>\delta\Big\}$. By (\ref{eq lemma3 06})
\begin{eqnarray*}
\lim_{\epsilon\rightarrow 0}\mathbb{E}\Big(\lambda_T(A_{\delta,\epsilon})\Big)
\leq
\frac{1}{\delta^{2m}}\lim_{\epsilon\rightarrow0}\mathbb{E}\int_0^T
\|X^\epsilon_t-X_t\|^{2m}_Hdt=0.
\end{eqnarray*}
Therefore, there exists a subsequence $\epsilon_k$ (for simplicity, we still denote it by the same notation $\epsilon_k$) such that
\begin{eqnarray}\label{eq recall}
\lim_{\epsilon_k\rightarrow 0}\lambda_T(A_{\delta,\epsilon_k})=0,\ \ \ \bar{\mathbb{P}}\text{-a.s.}.
\end{eqnarray}

Applying Lemma \ref{Lemma-Condition-0,H-1,H}, we have
\begin{eqnarray}\label{eq P21 01}
&&\int_0^T\int_{\mathbb{X}}\|f(s,X^{\epsilon_k}_s,z)-f(s,X_s,z)\|_H|\varphi_{\epsilon_k}(s,z)-1|\nu(dz)ds\nonumber\\
&\leq&
\int_0^T\int_{\mathbb{X}}G_f(s,z)\|X^{\epsilon_k}_s-X_s\|_H|\varphi_{\epsilon_k}(s,z)-1|\nu(dz)ds\nonumber\\
&\leq&
  \delta\int_{A^c_{\delta,\epsilon_k}}\int_{\mathbb{X}}G_f(s,z)
  |\varphi_{\epsilon_k}(s,z)-1|\nu(dz)ds\nonumber\\
&&
  +
  \sup_{s\in[0,T]}\|X^{\epsilon_k}_s-X_s\|_H\int_{A_{\delta,\epsilon_k}}\int_{\mathbb{X}}G_f(s,z)|\varphi_{\epsilon_k}(s,z)-1|\nu(dz)ds\nonumber\\
&\leq&
 \delta\sup_{\varphi\in S^N}\int_0^T\int_{\mathbb{X}}G_f(s,z)|\varphi(s,z)-1|\nu(dz)ds\nonumber\\
&& +
 \sup_{s\in[0,T]}\|X^{\epsilon_k}_s-X_s\|_H\sup_{\varphi\in S^N}\int_{A_{\delta,\epsilon_k}}\int_{\mathbb{X}}G_f(s,z)|\varphi(s,z)-1|\nu(dz)ds\\
&\leq&
  \delta C_{G_f,N}
  +
  \sup_{s\in[0,T]}\|X^{\epsilon_k}_s-X_s\|_H\sup_{\varphi\in S^N}\int_{A_{\delta,\epsilon_k}}\int_{\mathbb{X}}G_f(s,z)|\varphi(s,z)-1|\nu(dz)ds.\nonumber
\end{eqnarray}
Notice that
\begin{eqnarray}\label{eq P21 02}
  &&\mathbb{E}\Big(\sup_{s\in[0,T]}\|X^{\epsilon_k}_s-X_s\|_H\sup_{\varphi\in S^N}\int_{A_{\delta,\epsilon_k}}\int_{\mathbb{X}}G_f(s,z)|\varphi(s,z)-1|\nu(dz)ds\Big)\nonumber\\
  &\leq&
  \Big(\mathbb{E}\Big(\sup_{s\in[0,T]}\|X^{\epsilon_k}_s-X_s\|^2_H\Big)^{\frac12}
  \Big(\mathbb{E}\Big(\sup_{\varphi\in S^N}\int_{A_{\delta,\epsilon_k}}\int_{\mathbb{X}}G_f(s,z)
  |\varphi(s,z)-1|\nu(dz)ds\Big)^2\Big)^{\frac12}.
\end{eqnarray}
By the dominated convergence theorem, Lemma \ref{lem-thm2-02} c) and Lemma \ref{Lemma-Condition-0,H-1,H}, we have
\begin{eqnarray}\label{eq A 00}
  \lim_{{\epsilon_k}\rightarrow 0}\mathbb{E}\Big(\sup_{\varphi\in S^N}\int_{A_{\delta,\epsilon_k}}\int_{\mathbb{X}}G_f(s,z)|\varphi(s,z)-1|\nu(dz)ds\Big)^2=0.
\end{eqnarray}

Hence, (\ref{eq lemma3 01}), (\ref{eq lemma3 03}), (\ref{eq P21 01})-(\ref{eq A 00}) imply
\begin{eqnarray}\label{eq P21 03}
  \lim_{{\epsilon_k}\rightarrow0}\mathbb{E}\Big(\int_0^T\int_{\mathbb{X}}\|f(s,X^{\epsilon_k}_s,z)-f(s,X_s,z)\|_H|\varphi_{\epsilon_k}(s,z)-1|\nu(dz)ds\Big)=0.
\end{eqnarray}
So, there exists a subsequence $\epsilon_k$ (for simplicity, we still denote it by the same notation $\epsilon_k$) such that
$$
\lim_{\epsilon_k\rightarrow0}\int_0^T\int_{\mathbb{X}}\|f(s,X^{\epsilon_k}_s,z)-f(s,X_s,z)\|_H|\varphi_{\epsilon_k}(s,z)-1|\nu(dz)ds=0,\ \ \bar{\mathbb{P}}\text{-}a.s..
$$
Combining this with (\ref{eq star2}), we arrive at (\ref{eq star1}).
\end{proof}

Define
\begin{eqnarray}\label{eq step 5 01}
\widetilde{X}_t:=x+\int_0^tY_sds+\int_0^t\int_{\mathbb{X}}f(s,X_s,z)(\varphi(s,z)-1)\nu(dz)ds.
\end{eqnarray}
By taking weak limit of (\ref{eq SPDE 02}), it is not difficulty to see that
$$
\widetilde{X}_t(\omega)=\bar{X}_t(\omega)=X_t(\omega),\text{\ for\ }dt\times \bar{\mathbb{P}}\text{-almost\ all\ }(t,\omega).
$$

Set
$$
 \mathcal{N}:=\Big\{\phi:\ \phi\ \text{is a } V\text{-valued}\ \bar{\mathcal{F}}_t\text{-adapted process such that }\mathbb{E}\Big(\int_0^T\rho(\phi_s)ds\Big)<\infty\Big\}.
$$
Fix $\phi\in \mathcal{K}\cap\mathcal{N}\cap L^\infty([0,T],L^{\beta+2}(\Omega,H))$
and $\psi\in L^\infty([0,T],\RR)$. Denote
\begin{eqnarray*}
  G(X,\varphi,Y)
&  :=&
           \mathbb{E}\Big[\int_0^T\psi_t
             \int_0^te^{-\int_0^s(K_l+\rho(\phi_l))dl}\\
&&\qquad\times
                     2\left<\int_{\mathbb{X}}f(s,X_s,z)(\varphi(s,z)-1)\nu(dz),
                     Y_s\right>_{H,H}ds dt
                                    \Big].
\end{eqnarray*}

The following limiting result will be needed later.
\begin{lem}\label{lem0517a3}
\begin{eqnarray}\label{eq star 3}
\lim_{{\epsilon_k}\rightarrow0}G(X^{\epsilon_k},\varphi_{\epsilon_k},X^{\epsilon_k})=G(X,\varphi,X).
\end{eqnarray}
\end{lem}

\begin{proof}
For any fixed $(t,\omega)\in[0,T]\times\Omega$. Set
$$
\zeta(s,z)=\psi_te^{-\int_0^s(K_l+\rho(\phi_l))dl}\langle f(s,X_s,z),X_s\rangle_{H,H}.
$$
By Lemma \ref{lem-thm2-02} and $\sup_{s\in[0,T]}\|X_s\|_H<\infty$ $\bar{\mathbb{P}}$-a.s., we have $\forall (t,\omega)\in[0,T]\times\Omega$,
\begin{eqnarray*}
 &&\lim_{{\epsilon_k}\rightarrow0}
    \psi_t
             \int_0^te^{-\int_0^s(K_l+\rho(\phi_l))dl}
                        2\left<\int_{\mathbb{X}}f(s,X_s,z)
                        (\varphi_{\epsilon_k}(s,z)-1)\nu(dz),X_s\right>_{H,H}
                  ds\nonumber\\
 &=&
   \psi_t
            \int_0^te^{-\int_0^s(K_l+\rho(\phi_l))dl}
                   2\left<\int_{\mathbb{X}}f(s,X_s,z)(\varphi(s,z)-1)\nu(dz),
                   X_s\right>_{H,H}
                  ds.
\end{eqnarray*}
On the other hand, by Lemma \ref{Lemma-Condition-0,H-1,H}
\begin{eqnarray*}
  &&\sup_{\varphi\in S^N}
               \Big|
                  \psi_t\Big(\int_0^te^{-\int_0^s(K_l+\rho(\phi_l))dl}
                  \Big(
                     2\langle\int_{\mathbb{X}}f(s,X_s,z)(\varphi(s,z)-1)\nu(dz),X_s\rangle_{H,H}
                  \Big)ds
               \Big|\nonumber\\
  &\leq&
    C_{\psi}\sup_{\varphi\in S^N}\int_0^T
                     \int_{\mathbb{X}}\|f(s,X_s,z)\|_H\|X_s\|_H|\varphi(s,z)-1|\nu(dz)
                  ds\nonumber\\
  &\leq&
    C_{\psi}(1+\sup_{s\in[0,T]}\|X_s\|_H)^2\sup_{\varphi\in S^N}\int_0^T
                     \int_{\mathbb{X}}L_f(s,z)|\varphi(s,z)-1|\nu(dz)
                  ds\nonumber\\
  &\leq&
    C_{\psi,L_f,N}(1+\sup_{s\in[0,T]}\|X_s\|_H)^2.
\end{eqnarray*}
By the dominated convergence theorem, we have
\begin{eqnarray}\label{eq starstar1}
\lim_{{\epsilon_k}\rightarrow0}G(X,\varphi_{\epsilon_k},X)=G(X,\varphi,X).
\end{eqnarray}

Let $\delta>0$. Recall
$$
A_{\delta,{\epsilon_k}}:=\Big\{
                       s\in[0,T]:\ \|X^{\epsilon_k}_s-X_s\|_H>\delta
                     \Big\},
$$
and (\ref{eq recall}) that is there exists a subsequence $\epsilon_k$ such that
$$
  \lim_{\epsilon_k\rightarrow0}\lambda_T(A_{\delta,{\epsilon_k}})=0,\ \ \mathbb{P}\text{-a.s..}
$$
Then we have
\begin{eqnarray}\label{eq A1}
 &&\Big|G(X^{\epsilon_k},\varphi_{\epsilon_k},X^{\epsilon_k})-G(X^{\epsilon_k},\varphi_{\epsilon_k},X)\Big|\nonumber\\
 &\leq&
    C\mathbb{E}\Big(\int_0^T\int_{\mathbb{X}}\|f(s,X^{\epsilon_k}_s,z)\|_H|\varphi_{\epsilon_k}(s,z)-1|\|X^{\epsilon_k}_s-X_s\|_H\nu(dz)ds\Big)\nonumber\\
 &\leq&
   C\mathbb{E}\Big(\int_0^T\int_{\mathbb{X}}L_f(s,z)(1+\|X^{\epsilon_k}_s\|_H)|\varphi_{\epsilon_k}(s,z)-1|\|X^{\epsilon_k}_s-X_s\|_H\nu(dz)ds\Big)\nonumber\\
 &\leq&
   C\delta
    \mathbb{E}\Big(\int_{A_{\delta,{\epsilon_k}}^c}\int_{\mathbb{X}}L_f(s,z)(1+\|X^{\epsilon_k}_s\|_H)|\varphi_{\epsilon_k}(s,z)-1|\nu(dz)ds\Big)\nonumber\\
   &&+
   C\mathbb{E}
   \Big(\int_{A_{\delta,{\epsilon_k}}}\int_{\mathbb{X}}L_f(s,z)(1+\|X^{\epsilon_k}_s\|_H)|\varphi_{\epsilon_k}(s,z)-1|\|X^{\epsilon_k}_s-X_s\|_H\nu(dz)ds\Big)\nonumber\\
 &\leq&
   C\delta
    \mathbb{E}\Big(\sup_{s\in[0,T]}(1+\|X^{\epsilon_k}_s\|_H)\Big)
      \sup_{\varphi\in S^N}\int_0^T\int_{\mathbb{X}}L_f(s,z)|\varphi(s,z)-1|\nu(dz)ds\nonumber\\
      &&+
      C\mathbb{E}\Big[
                      \sup_{s\in[0,T]}\Big((1+\|X^{\epsilon_k}_s\|_H)
                      (\|X^{\epsilon_k}_s-X_s\|_H)\Big)\nonumber\\
&&\qquad\qquad\times                        \sup_{\varphi\in S^N}\int_{A_{\delta,{\epsilon_k}}}\int_{\mathbb{X}}L_f(s,z)|\varphi(s,z)-1|\nu(dz)ds
                 \Big]\nonumber\\
 &\leq&
   \delta C_{L_f,N}
   +
   C\Big(\mathbb{E}\Big(1+\sup_{s\in[0,T]}\|X^{\epsilon_k}_s\|_H\Big)^4\Big)^{1/4}
   \Big(\mathbb{E}\Big(1+\sup_{s\in[0,T]}\|X^{\epsilon_k}_s-X_s\|_H\Big)^4\Big)^{1/4}\nonumber\\
   &&\cdot
   \Big(\mathbb{E}\Big(\sup_{\varphi\in S^N}\int_{A_{\delta,{\epsilon_k}}}\int_{\mathbb{X}}L_f(s,z)|\varphi(s,z)-1|\nu(dz)ds\Big)^2\Big)^{1/2}.
\end{eqnarray}
Similar as (\ref{eq A 00}) and (\ref{eq P21 03}), we have
\begin{eqnarray}\label{eq A}
 \lim_{{\epsilon_k}\rightarrow0}\Big|G(X^{\epsilon_k},\varphi_{\epsilon_k},X^{\epsilon_k})-G(X^{\epsilon_k},\varphi_{\epsilon_k},X)\Big|=0.
\end{eqnarray}

On the other hand,
\begin{eqnarray*}
&&\Big|G(X^{\epsilon_k},\varphi_{\epsilon_k},X)-G(X,\varphi_{\epsilon_k},X)\Big|\nonumber\\
&\leq&
      C\mathbb{E}\Big(
           \int_0^T\int_{\mathbb{X}}\|f(s,X^{\epsilon_k}_s,z)-f(s,X_s,z)\|_H|\varphi_{\epsilon_k}(s,z)-1|\|X_s\|_H\nu(dz)ds
                 \Big)\nonumber\\
&\leq&
    C\mathbb{E}\Big(
           \int_0^T\int_{\mathbb{X}}G_f(s,z)\|X^{\epsilon_k}_s-X_s\|_H|\varphi_{\epsilon_k}(s,z)-1|\|X_s\|_H\nu(dz)ds
                 \Big).\nonumber
\end{eqnarray*}
Using the similar arguments as proving (\ref{eq A}), we have
\begin{eqnarray}\label{eq B}
  \lim_{{\epsilon_k}\rightarrow0}\Big|G(X^{\epsilon_k},\varphi_{\epsilon_k},X)-G(X,\varphi_{\epsilon_k},X)\Big|=0.
\end{eqnarray}

Combining (\ref{eq A}), (\ref{eq B}), and (\ref{eq starstar1}), we have (\ref{eq star 3}).
\end{proof}

\begin{lem}\label{lem0517a4}
$$
Y_t(\omega)=\mathcal{A}(t,X_t(\omega))\text{\ for\ }dt\times \bar{\mathbb{P}}\text{-almost\ all\ }(t,\omega).
$$
\end{lem}

\begin{proof}

For $\phi\in \mathcal{K}\cap\mathcal{N}\cap L^\infty([0,T],L^{\beta+2}(\Omega,H))$, applying the ${\rm It\hat{o}'s}$ formula,
\begin{eqnarray*}
  && e^{-\int_0^t(K_s+\rho(\phi_s))ds}\|X^{\epsilon_k}_t\|^2_H-\|x\|^2_H\nonumber\\
 &=&
   \int_0^te^{-\int_0^s(K_l+\rho(\phi_l))dl}
               \Big[
                 -(K_s+\rho(\phi_s))\|X^{\epsilon_k}_s\|^2_H
                 +
                 2\langle \mathcal{A}(s,X^{\epsilon_k}_s),X^{\epsilon_k}_s\rangle_{V^*,V}\nonumber\\
                 &&\ \ \ \ \ \ \ \ \ \ \ \ \ \ \ \ \ \ \ \ \ \ \ \ \ \ \ \ \ +
                 2\langle\int_{\mathbb{X}}f(s,X^{\epsilon_k}_s,z)(\varphi_{\epsilon_k}(s,z)-1)\nu(dz),X^{\epsilon_k}_s\rangle_{H,H}
               \Big]ds\nonumber\\
   &&+
   \int_0^te^{-\int_0^s(K_l+\rho(\phi_l))dl}\int_{\mathbb{X}}\Big[2{\epsilon_k}\langle f(s,X^{\epsilon_k}_{s-},z),X^{\epsilon_k}_{s-}\rangle_{H,H}\Big]\tilde{N}^{{\epsilon_k}^{-1}\varphi_{\epsilon_k}}(ds,dz)\nonumber\\
   &&+
   \int_0^te^{-\int_0^s(K_l+\rho(\phi_l))dl}\int_{\mathbb{X}}\Big[{\epsilon_k}^2\|f(s,X^{\epsilon_k}_{s-},z)\|^2_H\Big]N^{{\epsilon_k}^{-1}\varphi_{\epsilon_k}}(ds,dz).
\end{eqnarray*}
Notice that
$$
M_{\epsilon_k}(t):=\int_0^te^{-\int_0^s(K_l+\rho(\phi_l))dl}\int_{\mathbb{X}}\Big[2{\epsilon_k}\langle f(s,X^{\epsilon_k}_{s-},z),X^{\epsilon_k}_{s-}\rangle_{H,H}\Big]\tilde{N}^{{\epsilon_k}^{-1}\varphi_{\epsilon_k}}(ds,dz)
$$
is a square integrable martingale, we have
\begin{eqnarray}\label{eq star 0}
&& \mathbb{E}\Big(e^{-\int_0^t(K_s+\rho(\phi_s))ds}\|X^{\epsilon_k}_t\|^2_H\Big)-\|x\|^2_H\nonumber\\
 &=&
   -\mathbb{E}\Big(\int_0^te^{-\int_0^s(K_l+\rho(\phi_l))dl}(K_s+\rho(\phi_s))
                  \Big(\|X^{\epsilon_k}_s-\phi_s\|^2_H+2\langle X^{\epsilon_k}_s,\phi_s\rangle_{H,H}-\|\phi_s\|^2_H\Big)ds
             \Big)\nonumber\\
   &&+
   \mathbb{E}\Big(\int_0^te^{-\int_0^s(K_l+\rho(\phi_l))dl}
                  \Big(
                     2\langle \mathcal{A}(s,X^{\epsilon_k}_s)-\mathcal{A}(s,\phi_s),X^{\epsilon_k}_s-\phi_s\rangle_{V^*,V}\nonumber\\
                      &&\ \ \ \ \ \ \ \ \ +
                     2\langle \mathcal{A}(s,\phi_s),X^{\epsilon_k}_s-\phi_s\rangle_{V^*,V}
                      +
                     2\langle \mathcal{A}(s,X^{\epsilon_k}_s),\phi_s\rangle_{V^*,V}
                  \Big)ds
             \Big)\nonumber\\
   &&+
   \mathbb{E}\Big(\int_0^te^{-\int_0^s(K_l+\rho(\phi_l))dl}
                  \Big(
                     2\langle\int_{\mathbb{X}}f(s,X^{\epsilon_k}_s,z)(\varphi_{\epsilon_k}(s,z)-1)\nu(dz),X^{\epsilon_k}_s\rangle_{H,H}
                  \Big)ds
             \Big)\nonumber\\
   &&+
   \mathbb{E}\Big(
     {\epsilon_k} \int_0^te^{-\int_0^s(K_l+\rho(\phi_l))dl}\int_{\mathbb{X}}\|f(s,X^{\epsilon_k}_s,z)\|^2_H\varphi_{\epsilon_k}(s,z)\nu(dz)ds
             \Big)\nonumber\\
   &\leq&
     -\mathbb{E}\Big(\int_0^te^{-\int_0^s(K_l+\rho(\phi_l))dl}(K_s+\rho(\phi_s))
                  \Big(2\langle X^{\epsilon_k}_s,\phi_s\rangle_{H,H}-\|\phi_s\|^2_H\Big)ds
             \Big)\nonumber\\
   &&+
   \mathbb{E}\Big(\int_0^te^{-\int_0^s(K_l+\rho(\phi_l))dl}
                  \Big(
                     2\langle \mathcal{A}(s,\phi_s),X^{\epsilon_k}_s-\phi_s\rangle_{V^*,V}
                      +
                     2\langle \mathcal{A}(s,X^{\epsilon_k}_s),\phi_s\rangle_{V^*,V}
                  \Big)ds
             \Big)\nonumber\\
   &&+
   \mathbb{E}\Big(\int_0^te^{-\int_0^s(K_l+\rho(\phi_l))dl}
                  \Big(
                     2\langle\int_{\mathbb{X}}f(s,X^{\epsilon_k}_s,z)(\varphi_{\epsilon_k}(s,z)-1)\nu(dz),X^{\epsilon_k}_s\rangle_{H,H}
                  \Big)ds
             \Big)\nonumber\\
   &&+
   \mathbb{E}\Big(
     {\epsilon_k} \int_0^te^{-\int_0^s(K_l+\rho(\phi_l))dl}\int_{\mathbb{X}}\|f(s,X^{\epsilon_k}_s,z)\|^2_H\varphi_{\epsilon_k}(s,z)\nu(dz)ds
             \Big).
\end{eqnarray}
By (i) of Lemma \ref{lem0517a1}, we get
\begin{eqnarray}\label{eq star 1}
   && \mathbb{E}\Big[\int_0^T\psi_t\Big(e^{-\int_0^t(K_s+\rho(\phi_s))ds}\|X_t\|^2_H-\|x\|^2_H\Big)dt\Big]\nonumber\\
 &\leq&
    \liminf_{{\epsilon_k}\rightarrow0}
        \mathbb{E}\Big[\int_0^T\psi_t\Big(e^{-\int_0^t(K_s+\rho(\phi_s))ds}\|X^{\epsilon_k}_t\|^2_H-\|x\|^2_H\Big)dt\Big].
\end{eqnarray}

By Lemma \ref{Lemma-Condition-0,H-1,H},
\begin{eqnarray}\label{eq P26 star 2}
  &&\mathbb{E}\Big(
     {\epsilon_k} \int_0^te^{-\int_0^s(K_l+\rho(\phi_l))dl}\int_{\mathbb{X}}\|f(s,X^{\epsilon_k}_s,z)\|^2_H\varphi_{\epsilon_k}(s,z)\nu(dz)ds
             \Big)\nonumber\\
  &\leq&
  \mathbb{E}\Big(
     {\epsilon_k} \int_0^t\int_{\mathbb{X}}(1+\|X^{\epsilon_k}_s\|_H)^2L^2_f(s,z)\varphi_{\epsilon_k}(s,z)\nu(dz)ds
             \Big)\nonumber\\
  &\leq&
  {\epsilon_k}\mathbb{E}\Big((1+\sup_{s\in[0,T]}\|X^{\epsilon_k}_s\|_H)^2\Big)
           \sup_{\varphi\in S^N}\int_0^T\int_{\mathbb{X}}L^2_f(s,z)\varphi(s,z)\nu(dz)ds
             \nonumber\\
  &\leq&
   {\epsilon_k} C_{L_f,2,2,N}.
\end{eqnarray}

Combining from (\ref{eq star 0}) to (\ref{eq P26 star 2}), and Lemma \ref{lem0517a3}, we infer
\begin{eqnarray}\label{eq star 4}
 && \mathbb{E}\Big[\int_0^T\psi_t\Big(e^{-\int_0^t(K_s+\rho(\phi_s))ds}
 \|X_t\|^2_H-\|x\|^2_H\Big)dt\Big]\\
 &\leq&
   -\mathbb{E}\Big[\int_0^T\psi_t\int_0^te^{-\int_0^s(K_l+\rho(\phi_l))dl}(K_s+\rho(\phi_s))
                  \Big(2\langle X_s,\phi_s\rangle_{H,H}-\|\phi_s\|^2_H\Big)ds
             dt\Big]\nonumber\\
   &&+
   \mathbb{E}\Big[\int_0^T\psi_t\int_0^te^{-\int_0^s(K_l+\rho(\phi_l))dl}
                  \Big(
                     2\langle \mathcal{A}(s,\phi_s),X_s-\phi_s\rangle_{V^*,V}
                      +
                     2\langle Y_s,\phi_s\rangle_{V^*,V}
                  \Big)ds
             dt\Big]\nonumber\\
   &&+
   \mathbb{E}\Big[\int_0^T\psi_t\int_0^te^{-\int_0^s(K_l+\rho(\phi_l))dl}
                  \Big(
                     2\langle\int_{\mathbb{X}}f(s,X_s,z)(\varphi(s,z)-1)\nu(dz),X_s\rangle_{H,H}
                  \Big)ds
             dt\Big].\nonumber
\end{eqnarray}

On the other hand, by (\ref{eq step 5 01}), we have
\begin{eqnarray}\label{eq star 5}
  && \mathbb{E}\Big(e^{-\int_0^t(K_s+\rho(\phi_s))ds}\|X_t\|^2_H-\|x\|^2_H\Big)\\
 &=&
   -\mathbb{E}\Big(\int_0^te^{-\int_0^s(K_l+\rho(\phi_l))dl}(K_s+\rho(\phi_s))
                  \|X_s\|^2_Hds
             \Big)\nonumber\\
   &&+
   \mathbb{E}\Big(\int_0^te^{-\int_0^s(K_l+\rho(\phi_l))dl}
                     2\langle Y_s,X_s\rangle_{V^*,V}
                  ds
             \Big)\nonumber\\
   &&+
   \mathbb{E}\Big(\int_0^te^{-\int_0^s(K_l+\rho(\phi_l))dl}
                  \Big(
                     2\langle\int_{\mathbb{X}}f(s,X_s,z)(\varphi(s,z)-1)\nu(dz),X_s\rangle_{H,H}
                  \Big)ds
             \Big).\nonumber
\end{eqnarray}
By (\ref{eq star 4}) and (\ref{eq star 5}), we have
\begin{eqnarray*}
  &&\mathbb{E}\Big[\int_0^T\psi_t\int_0^te^{-\int_0^s(K_l+\rho(\phi_l))dl}
                   \Big(-(K_s+\rho(\phi_s))\|X_s-\phi_s\|^2_H\\
                   &&\qquad\qquad+2\langle \mathcal{A}(s,\phi_s)-Y_s,X_s-\phi_s\rangle_{V^*,V}\Big)
                   \Big)ds
             dt\Big]
  \leq
    0.
\end{eqnarray*}
Put $\phi=X-\eta \tilde{\phi}v$ for $\tilde{\phi}\in L^\infty([0,T]\times\Omega;dt\times\bar{\mathbb{P}};\mathbb{R})$ and $v\in V$,
divide both sides by $\eta$ and let $\eta\rightarrow0$, then we have
\begin{eqnarray*}
  \mathbb{E}\Big[\int_0^T\psi_t\int_0^te^{-\int_0^s(K_l+\rho(\phi_l))dl}
                   \Big(2\tilde{\phi}_s\langle \mathcal{A}(s,\phi_s)-Y_s,v\rangle_{V^*,V}\Big)
                   \Big)ds
             dt\Big]
  \leq
    0.
\end{eqnarray*}

Hence $Y=\mathcal{A}(\cdot,X)$.
\end{proof}

\begin{prop}\label{lem 4.4}
$X(\omega)$ solves the
following equation:
\begin{eqnarray}\label{eq limit X 01}
X_t(\omega)=x+\int_0^t\mathcal{A}(s,X_s(\omega))ds+\int_0^t\int_{\mathbb{X}}f(s,X_s(\omega),z)(\varphi(s,z)(\omega)-1)\nu(dz)ds,
\end{eqnarray}
which has an unique solution in $C([0,T],H)\cap L^\alpha([0,T],V).$
\end{prop}

\begin{proof}
The equation (\ref{eq limit X 01}) follows from Lemmas \ref{lem0517a1}-\ref{lem0517a4}. The proof of the uniqueness is standard, and it is omitted.
\end{proof}

\begin{lem}\label{lem 5.5}
There exists a subsequence $\varpi_k$, such that
\begin{eqnarray}\label{eq limit X 02}
\lim_{\varpi_k\rightarrow0}\sup_{t\in[0,T]}\|X^{\varpi_k}_t-X_t\|^2_H=0,\ \ \ \bar{\mathbb{P}}\text{-a.s.}.
\end{eqnarray}
\end{lem}

\begin{proof}

Set $L^{\epsilon_k}_t=X^{\epsilon_k}_t-X_t$. Then
\begin{eqnarray}\label{eq I}
  &&e^{-\int_0^t(K_s+\rho(X_s))ds}\|L^{\epsilon_k}_t\|^2_H\nonumber\\
 &=&
  \int_0^t e^{-\int_0^s(K_r+\rho(X_r))dr}
       \Big(
            -(K_s+\rho(X_s))\|L^{\epsilon_k}_s\|^2_H\nonumber\\
&&\qquad\qquad \qquad\qquad\qquad\qquad           +
            2\langle \mathcal{A}(s,X^{\epsilon_k}_s)-\mathcal{A}(s,X_s), L^{\epsilon_k}_s\rangle_{V^*,V}
       \Big) ds\nonumber\\
  &&+
    2\int_0^t e^{-\int_0^s(K_r+\rho(X_r))dr}
       \Big< \int_{\mathbb{X}}f(s, X^{\epsilon_k}_s,z)(\varphi_{\epsilon_k}(s,z)-1)\nu(dz)\nonumber\\
&&\qquad\qquad\qquad\qquad\qquad\qquad               -
               \int_{\mathbb{X}}f(s, X_s,z)(\varphi(s,z)-1)\nu(dz),
               L^{\epsilon_k}_s
       \Big>_{H,H}ds\nonumber\\
  &&+
   2{\epsilon_k}\int_0^t e^{-\int_0^s(K_r+\rho(X_r))dr}
               \langle \int_{\mathbb{X}}f(s, X^{\epsilon_k}_s,z),L^{\epsilon_k}_s\rangle_{H,H}
               \tilde{N}^{{\epsilon_k}^{-1}\varphi_{\epsilon_k}}(dz,ds)\nonumber\\
  &&+
    {\epsilon_k}^2\int_0^t e^{-\int_0^s(K_r+\rho(X_r))dr}\int_{\mathbb{X}}
               \|f(s, X^{\epsilon_k}_s,z)\|^2_{H}
               N^{{\epsilon_k}^{-1}\varphi_{\epsilon_k}}(dz,ds)\nonumber\\
  &=&I_1(t)+I_2(t)+I_3(t)+I_4(t).
\end{eqnarray}
(H2) implies
\begin{eqnarray}\label{eq I1}
I_1(t)\leq 0.
\end{eqnarray}
By (\ref{eq A1}) and (\ref{eq A}), we have
\begin{eqnarray}\label{eq I2-1}
&&\mathbb{E}\Big(
              \sup_{t\in[0,T]}\Big|
                                   \int_0^t e^{-\int_0^s(K_r+\rho(X_r))dr}
       \langle \int_{\mathbb{X}}f(s, X^{\epsilon_k}_s,z)(\varphi_{\epsilon_k}(s,z)-1)\nu(dz),
               L^{\epsilon_k}_s
       \rangle_{H,H}ds
                              \Big|
          \Big)\nonumber\\
&\leq&
   \mathbb{E}\Big(
       \int_0^T\int_{\mathbb{X}}
           \|f(s, X^{\epsilon_k}_s,z)\|_H|\varphi_{\epsilon_k}(s,z)-1|\|L^{\epsilon_k}_s\|_H
          \nu(dz)ds
          \Big)\nonumber\\
&\leq&
   \mathbb{E}\Big(
       \int_0^T\int_{\mathbb{X}}
           \|X^{\epsilon_k}_s\|_HL_f(s,z)|\varphi_{\epsilon_k}(s,z)-1|\|L^{\epsilon_k}_s\|_H
          \nu(dz)ds
          \Big)\rightarrow 0, \text{as}\ {\epsilon_k}\rightarrow0.
\end{eqnarray}
Then it is not difficulty to obtain
\begin{eqnarray}\label{eq I2}
\lim_{{\epsilon_k}\rightarrow0}\mathbb{E}\Big(\sup_{t\in[0,T]}|I_2(t)|\Big)=0.
\end{eqnarray}

For $I_3$,
\begin{eqnarray}\label{eq I3}
 &&\mathbb{E}\Big(
             \sup_{t\in[0,T]}|I_3(t)|
           \Big)\nonumber\\
 &\leq&
   \mathbb{E}\Big(
             \int_0^T\int_{\mathbb{X}}
                 4{\epsilon_k}^2\|L^{\epsilon_k}_s\|^2_H\|f(s,X^{\epsilon_k}_s,z)\|^2_H
                 N^{{\epsilon_k}^{-1}\varphi_{\epsilon_k}}(ds,dz)
             \Big)^{1/2}\nonumber\\
 &\leq&
      2\mathbb{E}\Big(
                  \sqrt{{\epsilon_k}}\sup_{s\in[0,T]}\|L^{\epsilon_k}_s\|_H
                  \Big(
                    \int_0^T\int_{\mathbb{X}}{\epsilon_k}\|f(s,X^{\epsilon_k}_s,z)\|_H^2N^{{\epsilon_k}^{-1}\varphi_{\epsilon_k}}(ds,dz)
                  \Big)^{1/2}
                 \Big)\nonumber\\
 &\leq&
     2\sqrt{{\epsilon_k}}
     \Big(\mathbb{E}\Big(\sup_{t\in[0,T]}\|L^{\epsilon_k}_t\|^2_H\Big)\Big)^{1/2}
     \Big(\mathbb{E}\Big(\int_0^T\int_{\mathbb{X}}\|f(s,X^{\epsilon_k}_s,z)\|_H^2\varphi_{\epsilon_k}(s,z)\nu(dz)ds\Big)\Big)^{1/2}\nonumber\\
 &\leq&
    2\sqrt{{\epsilon_k}}
     \Big(\mathbb{E}\Big(\sup_{t\in[0,T]}\|L^{\epsilon_k}_t\|^2_H\Big)\Big)^{1/2}\nonumber\\
&&\qquad\times     \Big(\mathbb{E}\Big(1+\sup_{t\in[0,T]}\|X^{\epsilon_k}_t\|^2_H\Big)\sup_{\varphi\in S^N}\int_0^T\int_{\mathbb{X}}L_f^2(s,z)\varphi(s,z)\nu(dz)ds\Big)^{1/2}\nonumber\\
 &&\rightarrow 0,\ \text{as}\ {\epsilon_k}\rightarrow0.
\end{eqnarray}
For $I_4$,
\begin{eqnarray}\label{eq I4}
&&\mathbb{E}\Big(\sup_{t\in[0,T]}|I_4(t)|\Big)\nonumber\\
&\leq&
{\epsilon_k} \mathbb{E}\Big(\int_0^T\int_{\mathbb{X}}\|f(s,X^{\epsilon_k}_s,z)\|^2_H\varphi_{\epsilon_k}(s,z)\nu(dz)ds\Big)\nonumber\\
&\leq&
{\epsilon_k} \mathbb{E}\Big(1+\sup_{t\in[0,T]}\|X^{\epsilon_k}_t\|^2_H\Big)\sup_{\varphi\in S^N}\int_0^T\int_{\mathbb{X}}L_f^2(s,z)\varphi(s,z)\nu(dz)ds\nonumber\\
&&\rightarrow 0,\ \text{as}\ {\epsilon_k}\rightarrow0.
\end{eqnarray}

Combining (\ref{eq I})--(\ref{eq I4}), we have
\begin{eqnarray*}
  &&\lim_{{\epsilon_k}\rightarrow0}\mathbb{E}\Big(\sup_{t\in[0,T]}\Big(e^{-\int_0^t(K_s+\rho(X_s))ds}\|L^{\epsilon_k}_t\|^2_H\Big)\Big)=0.
  \end{eqnarray*}
Then
\begin{eqnarray*}
  &&\lim_{{\epsilon_k}\rightarrow0}\mathbb{E}\Big(e^{-\int_0^T(K_s+\rho(X_s))ds}\Big(\sup_{t\in[0,T]}\|L^{\epsilon_k}_t\|^2_H\Big)\Big)=0.
  \end{eqnarray*}
This implies that there exists a subsequence $\varpi_k$ such that $X^{\varpi_k}$ converges to $X$ $\bar{\mathbb{P}}$-a.s..

\end{proof}

\section{Verification of Condition \ref{LDP}}
\setcounter{equation}{0}
\renewcommand{\theequation}{\thesection.\arabic{equation}}

Recall (\ref{define G-epsilon}) and (\ref{eq G0}), we have
\begin{theorem}\label{th con 2}
Fixed $N\in\mathbb{N}$, and let $\varphi_\epsilon,\varphi\in \tilde{\mathbb{A}}^N$ be such that $\varphi_\epsilon$ converges in distribution
to $\varphi$ as $\epsilon\rightarrow 0$. Then
$$
\mathcal{G}^\epsilon(\epsilon N^{\epsilon^{-1}\varphi_\epsilon})\Rightarrow \mathcal{G}^0(\nu_T^\varphi).
$$

\end{theorem}

\begin{proof} Recall $\bar{\mathbb{M}}$ in Section 2 and notations in Proposition \ref{lem 4.6}.
Denote
$$
\Pi=\Big(S^N,\ D([0,T],V^*),\ C([0,T],V^*),\ C([0,T],V^*),\ \bar{\mathbb{M}}\Big).
$$
Proposition \ref{lem 4.6} implies that  the laws of $\Big\{\Big(\varphi_\epsilon,M^\epsilon,Z^\epsilon,Y^\epsilon, \bar{N}\Big),\ \epsilon>0\Big\}$ is tight in $\Pi$. Let $\Big(\varphi, 0, Z, Y, \bar{N}\Big)$ be any limit point of the tight family. By the Skorohod's embedding theorem, there exist a stochastic basis $(\Omega^1,\mathcal{F}^1,\mathbb{P}^1)$ and, on this basis, $\Pi$-valued random variables $\Big(\overrightarrow{\varphi}_\epsilon,\overrightarrow{M}^\epsilon,\overrightarrow{Z}^\epsilon,\overrightarrow{Y}^\epsilon, \overrightarrow{N}_\epsilon\Big)$, $\Big(\overrightarrow{\varphi},0,\overrightarrow{Z},\overrightarrow{Y}, \overrightarrow{N}_0\Big)$, such that
$\Big(\overrightarrow{\varphi}_\epsilon,\overrightarrow{M}^\epsilon,\overrightarrow{Z}^\epsilon,\overrightarrow{Y}^\epsilon, \overrightarrow{N}_\epsilon\Big)$(respectively $\Big(\overrightarrow{\varphi},0,\overrightarrow{Z},\overrightarrow{Y}, \overrightarrow{N}_0\Big)$)
has the same law as $\Big(\varphi_\epsilon,M^\epsilon,Z^\epsilon,Y^\epsilon, \bar{N}\Big)$(respectively $\Big(\varphi, 0, Z, Y, \bar{N}\Big)$), and
$$
\Big(\overrightarrow{\varphi}_\epsilon,\overrightarrow{M}^\epsilon,\overrightarrow{Z}^\epsilon,\overrightarrow{Y}^\epsilon, \overrightarrow{N}_\epsilon\Big)
\longrightarrow
 \Big(\overrightarrow{\varphi},0,\overrightarrow{Z},\overrightarrow{Y}, \overrightarrow{N}_0\Big)\ \text{in}\ \Pi,\ \mathbb{P}^1\text{-}a.s..
$$
Set $\overrightarrow{X}^\epsilon=x+\overrightarrow{M}^\epsilon+\overrightarrow{Z}^\epsilon+\overrightarrow{Y}^\epsilon$ and
$\overrightarrow{X}=x+\overrightarrow{Z}+\overrightarrow{Y}$. From the equation satisfied by $\Big\{\Big(\varphi_\epsilon,M^\epsilon,Z^\epsilon,Y^\epsilon, \bar{N}\Big),\ \epsilon>0\Big\}$, we have that $\overrightarrow{X}^\epsilon$ satisfies
the following SPDE
\begin{eqnarray*}
d \overrightarrow{X}^\epsilon_t&=&\mathcal{A}(t,\overrightarrow{X}^\epsilon_t)dt
               +\int_{\mathbb{X}}f(t,\overrightarrow{X}^\epsilon_t,z)
               (\overrightarrow{\varphi}_\epsilon(t,z)-1)\nu(dz)dt\\
         &&      +\epsilon\int_{\mathbb{X}}f(t,\overrightarrow{X}^\epsilon_{t-},z)\widetilde{\overrightarrow{N}}_\epsilon^{\epsilon^{-1}\overrightarrow{\varphi}_\epsilon}(dz,dt),
\end{eqnarray*}
here $\overrightarrow{N}_\epsilon^{\varphi}$ is defined as (\ref{Jump-representation}), that is
   \begin{eqnarray*}
      \overrightarrow{N}_\epsilon^{\varphi}((0,t]\times U)=\int_{(0,t]\times U}\int_{(0,\infty)}1_{[0,\varphi(s,x)]}(r)\overrightarrow{N}_\epsilon(dsdxdr),\
      t\in[0,T],U\in\mathcal{B}(\mathbb{X}),
   \end{eqnarray*}
and $\widetilde{\overrightarrow{N}}_\epsilon^{\varphi}$ is the compensated Poisson random measure with respect to $\overrightarrow{N}_\epsilon^{\varphi}$.
\vskip 0.3cm

Using the fact that if $f_n\in D([0,T],\mathbb{R})$ and $\lim_{n\rightarrow\infty}f_n=0$ with the Skorokhod topology of $D([0,T],\mathbb{R})$, then
$\lim_{n\rightarrow\infty}\sup_{t\in[0,T]}|f_n(t)|=0$. We have
\[\lim_{\epsilon\rightarrow0}\sup_{t\in[0,T]}\|\overrightarrow{M}^\epsilon(t)\|_{V^*}=0, \qquad \mathbb{P}^1\mbox{-a.s.}.\]

Notice that \[\lim_{\epsilon\rightarrow0}\sup_{t\in[0,T]}\|\overrightarrow{Z}^\epsilon(t)
-\overrightarrow{Z}(t)\|_{V^*}=0,\qquad \mathbb{P}^1\mbox{-a.s.}\] and \[\lim_{\epsilon\rightarrow0}\sup_{t\in[0,T]}\|\overrightarrow{Y}^\epsilon(t)
-\overrightarrow{Y}(t)\|_{V^*}=0,\qquad \mathbb{P}^1\mbox{-a.s.},\] we have
$$
\lim_{\epsilon\rightarrow0}\sup_{t\in[0,T]}\|\overrightarrow{X}^\epsilon(t)-\overrightarrow{X}(t)\|_{V^*}=0,\ \mathbb{P}^1\text{-a.s..}
$$

Finally, following the proof of Proposition \ref{lem 4.4} and Lemma \ref{lem 5.5}, we can obtain $\overrightarrow{X}$ is the unique solution of (\ref{eq limit X 01}) with
$\varphi$ replaced by $\overrightarrow{\varphi}$, and there exists a subsequence $\varpi_k$ that
$$
\lim_{\varpi_k\rightarrow0}\sup_{t\in[0,T]}\|\overrightarrow{X}^{\varpi_k}(t)-\overrightarrow{X}(t)\|_H=0,\ \ \mathbb{P}^1\text{-a.s.}
$$
which implies this theorem.

\end{proof}

We have finished to verify the second part of Condition \ref{LDP}. To obtain the first part of Condition \ref{LDP},
we just need to replace $\epsilon\int_{\mathbb{X}}f(t,\widetilde{X}^\epsilon_{t-},z)\widetilde{N}^{\epsilon^{-1}\varphi_\epsilon}(dz,dt)$ by 0 in (\ref{eq SPDE 02}) and replacing $\varphi_\epsilon$ by deterministic elements $g_n$ in
 in the proof of Lemma \ref{lemma 4.1}--Lemma \ref{lem 4.4}, then we can similarly prove
the following result.
\begin{theorem}\label{th con 1}
Recall $\mathcal{G}^0$ in (\ref{eq G0}). For all $N\in\mathbb{N}$, let $g_n\rightarrow g$ as $n\rightarrow\infty$. Then
$$
\lim_{n\rightarrow\infty}\sup_{t\in[0,T]}\|\mathcal{G}^0(\nu_T^{g_n})(t)-\mathcal{G}^0(\nu_T^{g})(t)\|_H=0.
$$
\end{theorem}

\end{document}